\documentclass[12pt]{amsart}
\usepackage{amssymb,amsthm,amsmath,amstext}
\usepackage{mathrsfs}  
\usepackage{bm}        
\usepackage{mathtools} 

\usepackage[left=1.28in,top=.8in,right=1.28in,bottom=.8in]{geometry}

\usepackage{graphicx}
\usepackage[all]{xy}

\theoremstyle{plain}
\newtheorem{theorem}{Theorem}[section]
\newtheorem{prop}[theorem]{Proposition}
\newtheorem{lemma}[theorem]{Lemma}

\newtheorem{cor}[theorem]{Corollary}

\theoremstyle{definition}

\theoremstyle{remark}

\newcommand{\sheaf}[1]{\mathscr{#1}}

\renewcommand{\AA}{\sheaf{A}}

\newcommand{\DD}{\sheaf{D}}

\newcommand{\disc}{d}

\newcommand{\Z}{\mathbb Z}

\newcommand{\Q}{\mathbb Q}
\newcommand{\R}{\mathbb R}

\newcommand{\G}{\mathbb G}

\usepackage{hyperref}
\usepackage{color}

  \usepackage[OT2, T1]{fontenc}
\DeclareSymbolFont{cyrletters}{OT2}{wncyr}{m}{n}
\DeclareMathSymbol{\Sha}{\mathalpha}{cyrletters}{"58}

\hypersetup{colorlinks=true,linkcolor=blue,anchorcolor=blue,citecolor=blue,letterpaper=true}

\begin{document}

\title[ Rost injectivity]
{ Rost injectivity for   classical groups   over function fields of curves over local fields} 

\author[Parimala]{R.\ Parimala }
\address{Department of Mathematics  \\ %
Emory University \\ %
400 Dowman Drive~NE \\ %
Atlanta, GA 30322, USA}
\email{praman@emory.edu}

\author[Suresh]{V.\ Suresh}
\address{Department of Mathematics \\ %
Emory University \\ %
400 Dowman Drive~NE \\ %
Atlanta, GA 30322, USA}
\email{suresh.venapally@emory.edu}

\begin{abstract}  Let   $F_0$  be a complete discretely valued field with residue field  a 
global field or a local field with no real orderings. Let  $F/F_0$ be a  separable quadratic extension
and  $A$  a  central  simple algebra over $F$ with a $F/F_0$-involution $\tau$.
Let $h$ be a hermitian form over $(A,\tau)$ and $G = SU(A,\tau, h)$.
If  $2(ind(A))$ is coprime to  the characteristic of the residue field of $F_0$,  then we prove that the Rost invariant 
$R_G : H^1(F_0, G) \to H^3(F_0, \Q/\Z(2))$ is injective.    Let $K$ be a local field and $L$ the function field of a 
curve over $K$.   Let   $G$ be  an absolutely simple  simply connected 
linear algebraic group over $L$ of classical type. Suppose that the characteristic of the residue field of $K$ is a good prime for $G$.
As a consequence of our result and some known results we conclude   that   the Rost invariant of $G$ is injective. 
 \end{abstract}

 \maketitle
 
 \section{Introduction}
 
 Let $F$ be a field and $G$ an absolutely simple simply connected linear algebraic group defined over $F$. Rost 
 defines a degree three  cohomological invariant for $G$ torsors which yields a map 
 $$R_G : H^1(F, G) \to H^3(F, \Q/\Z((2)))$$
known as the  {\it Rost invariant}.   The triviality of the kernel  of $R_G$ for a given field has consequences
for the arithmetic of $F$; for instance,  if a torsor under $G$ over $F$ has a zero-cycle of degree 1, then it admits 
a rational point. 
 
 Let $A$ be a central simple algebra  over $F$ and $G = SL_1(A)$.  Then 
$H^1(F, G) \simeq F^*/Nrd(A^*)$ and $R_G : H^1(F, G)  \to H^3(F, \Q/\Z((2)))$ is given by $(a) \cdot [A]$ for all 
$a \in F^*$, where $[A]$ denotes the class of $A$ in $H^2(F, \Q/\Z(1))$ (cf. \cite[p. 437]{KMRT}). 
If the index$(A)$ is square free,  a result of Merkurjev and Suslin  (cf. \cite[Theorem 24.2]{suslin})
asserts that  the kernel of $R_G$ is trivial.
It is known that $R_G$ has trivial kernel if $G$ is of type $G_2$ (\cite[Theorem 5.4]{preeti}) or 
quasi split of type $^{3, 6}D_4$ (\cite{Chernousov}), 
$E_6$ or $E_7$ (\cite{Skip}, \cite{Chernousov}). 

 In general 
the kernel of $R_G$ need not be  trivial. In fact there are examples ( due to Merkurjev)  of  fields of cohomological dimension 
3 where the Rost invariant 
admits a nontrivial kernel (cf. \cite[Remark 5.1]{CTPS1}). However in these examples the underlying fields are  not finitely generated over global fields, 
local fields or finite fields. A   natural question arises as to whether $R_G$ has trivial kernel over function fields of curves (resp. surfaces)  over 
global fields or local fields (resp. finite fields). 
In this paper we discuss  this question for function fields of curves over local fields for classical groups. 

 Let $G$ be an  absolutely  simple simply connected linear algebraic group  of classical type over $F$. 
  We say that a  prime $p$ is {\it good} for $G$,  if  $p \neq 2$ for $G$  of type 
  $B_n$, $C_n$, $D_n$ ($D_4$ non trialitarian),  $p$ does not divide $n+1$ for $G$ of type $^1\!\!A_n$ and 
  $p $ does not divide $2(n+1)$ for  $G$ of type $^2\!\!A_n$.    
  
 Throughout this paper by a local field we mean a non-archimedain  local field (i.e. a complete discretely valued field with residue field a finite field). 
  
The main theorem of this paper is the following.

\begin{theorem} (\ref{main-rost})  Let $K$ be a local field with residue field $\kappa$ and
 $F$  the function field of a  curve over $K$. Let $G$ be an absolutely simple simply connected linear algebraic 
  group over $F$ of classical type.
 If char$(\kappa)$ is good with respect to $G$, then 
 $$R_G : H^1(F, G) \to H^3(F, \Q/\Z(2))$$
 is injective. 
\end{theorem} 

For groups  of type $^1A_n$,  the above theorem is  proved in (\cite{PPS}).
For groups of   type $B_n$, $C_n$ or $D_n$ ($D_4$ non trialitarian),  the above theorem is due to Preeti (\cite{preeti}).
If $G$ is any  quasisplit group over the function field of a curve over a local field,
 then the  triviality of the kernel of $R_G$   is proved in (\cite{CTPS1}). 
The main content of the paper is to prove the above theorem for groups of type $^2A_n$, i.e. special unitary groups 
of algebras with unitary involutions.  We note that the injectivity statement in the above theorem is equivalent to the 
triviality of the kernel of $R_G$ using a twisting argument.

We   begin by  proving  a  classification theorem for hermitian forms over division algebras with 
unitary involutions over complete discretely valued field with residue field a local field or a global field with no real places. 

For a hermitian form $h$  over central simple algebras with unitary involutions, let dim$(h)$ be the dimension and 
disc$(h)$ the discriminant of $h$ (cf. \S \ref{prelims}).  If dim$(h)$ is even and disc$(h)$ is trivial, 
let $R(h)$ be the Rost invariant of  $h$  (cf. \S \ref{rost}).

\begin{theorem} (\ref{class-complete}) Let $F$ be a complete discretely valued field with residue   field $K$ a global field with 
no real places or a local field.  Let $A$ be a central simple algebra over $F$ with 
 an involution  $\tau$ of second kind.  Suppose that $2per(A)$ is coprime to char$(K)$.   
 Let $h$  be a hermitian form over  $(A, \tau)$. 
 If dim$(h)$ is even,   disc$(h)$ is trivial   and $R(h)$ is trivial, then 
 $h$ is hyperbolic.  
 \end{theorem}  
 
 The above classification theorem reduces the triviality of the Rost kernel to the subset of elements 
 in $H^1(F_0, SU(A, \tau, h))$ which are in the image of  the connecting map
 $\delta : F^{*1} \to  H^1(F_0, SU(A, \tau, h))$ (cf. \S \ref{rost}).  Results from class field theory are used in this part of the proof. 
 Finally we arrive at the triviality of the Rost kernel for function fields of curves over local fields (\ref{main-rost-su}) by appealing to 
 a theorem of Kato on the Hasse principle for the degree three Galois cohomology  (\cite[Theorem 5.2]{Ka}) 
 and a Hasse principle for torsors 
 under $SU(A, \tau, h)$ proved in (\cite[Theorem 13.1]{PS2022}). 

We finally remark that Theorem 1.2 together  with  Hasse principle for hyperbolicity of 
hermitian forms (\cite[Theorem 11.6]{PS2022}) we get the following classification theorem for hermitian forms 
over the function fields of curve over local fields. 

\begin{theorem}(\ref{class-semiglobal}) Let $K$ be a local field with residue field $\kappa$ and
 $F$  the function field of a  curve over $K$. Let $A$ be a central simple algebra over $F$ with 
 an involution  $\tau$ of second kind. Suppose that 2per$(D)$ is coprime to char$(\kappa)$.
  Let $h$  be a hermitian form over  $(A, \tau)$. 
 If dim$(h)$ is even,   disc$(h)$ is trivial   and $R(h)$ is trivial, then 
 $h$ is hyperbolic.  
\end{theorem}

Here is an outline of the structure of the paper. Our main discussion is about the 
 special unitary groups of hermitian forms over central simple algebras with involutions of second kind. 
We begin by recalling some properties of  basic invariants of hermitian  forms in \S\ref{prelims}. 
Next we  recall the definition of  the Rost invariant of hermitian forms  in \S\ref{rost} and 
make some computations of the  Rost invariant for some classes of hermitian forms   in  \S\ref{rost-comp}. 
 In \S\ref{class-gen}, we give a classification of low rank hermitian forms over general fields and
recall a classification of hermitian forms over fields of cohomological dimension 3. 
Using global and local class field theory, we prove some results on 
hermitian forms and algebras with involutions of second kind over global fields  with no real places  and local fields in \S\ref{global-local}. 
 In \S\ref{Morita}, we describe the correspondence of hermitian forms  in certain set-up 
associated to  Morita equivalence.  In \S\ref{classification}, we give a classification of hermitian forms over complete discretely 
valued fields having residue fields global fields  with no real places  or local fields.  In \S\ref{main},  using the classification of
 results of \S \ref{classification}, we 
prove the  injectivity of  the Rost invariant  for  groups of type $^2A_n$ over function fields of curves over local fields.

We thank Merkurjev for  discussions on the Rost invariant. 

 \section{Preliminaries}
 \label{prelims}
 We refer the reader to (\cite{Sch},  \cite{Knus} and \cite{BP}) for generalities on hermitian forms. 
 Let $F_0$ be a field  of with char$(F_0) \neq 2$ and $F/F_0$ a field extension of degree at most 2. 
 Let $A$ be a central simple algebra over $F$ with a $F/F_0$-involution $\tau$. 
 Let $(V, h)$ (or simply $h$)  be a  hermitian form over $(A, \tau)$.   We 
 write  $A = M_n(D)$ for some division algebra $D$ over $F$.   The 
  {\it dimension}  of $h$,   denoted by dim$(h)$, is defined to be dim$_DV/n$. 
  
 Suppose $A = D$ is division.  If $m = dim(h)$,  then $h$ can be represented by  matrix $(a_{ij})$ for some 
$a_{ij} \in A$ with $\tau(a_{ij}) = a_{ji}$. 
The {\it discriminant } of $h$ defined as $disc(h) = (-1)^{m(m-1)/2} Nrd_A((a_{ij})) \in F_0^*$.
Suppose $A = M_n(D)$.  Then $D$ has an involution  $\tau_0$ of same kind as $\tau$
and $\tau$ is the adjoint involution given by a hermitian form over $(D, \tau_0)$. 
Let $\tilde{h}$ be the hermitian form  over $(D,\tau_0)$ associated to $h$ by Morita equivalence. 
If the dim$(h)$ is even, 
Then the {\it discriminant} of $h$ is defined as the discriminant of $\tilde{h}$ (\cite[\S 2.2]{BP}). 
If $\tau$ is of first kind (i.e. $\tau$ is identity on $F =F_0$), then disc$(h)$ is well defined  modulo $ F^{*2}$
and we consider  disc$(h)$ as an element in $ F^*/F^{*2}$. 
If $\tau$ is of second kind, then  disc$(h)$ is well defined modulo $N_{F/F_0}(F^*)$ and 
we consider disc$(h)$ as an element in $F_0^*/N_{F/F_0}(F^*)$.

 Suppose  $\tau$ is of second kind and deg$(A)dim(h)$ is even. We also have the discriminant algebra  $\DD(A,\tau, h)$ (\cite[Ch.II, \S 10]{KMRT}). 
 Suppose $A = F$. Then $h$ is given by diagonal form $<a_1, \cdots , a_m>$ for some $a_i \in F_0^*$.
 In this case disc$(h) = (-1)^{m(m-1)} a_1 \cdots a_m \in F_0^*/N_{F/F_0}(F^*)$. 
 If $F = F_0(\sqrt{\theta})$, then $\DD(A, \tau, h) $ is Brauer equivalent to the quaternion algebra $(\theta, disc(h))$
 (\cite[Proposition 10.33]{KMRT}). 
  
 \begin{prop}
 \label{even} Let $K_0$ be a field  and  $K/K_0$ a quadratic extension. 
 Let $A$ be  a central division algebra over $K$  of even degree  with a $K/K_0$-involution $\tau$. 
 Let $h$ be a hermitian form over $(A,\tau)$. Suppose that per$(A) = $ deg$(A)$.  
 If the discriminant algebra $\DD(A, \tau, h)$ of $(A, \tau, h)$ is trivial,
 then dim$(h)$ is even.  
 \end{prop}
 
 \begin{proof} Let deg$(A) = 2n$ and dim$(h) = m$. Then, by (\cite[ Proposition 10.30 and p. 116]{KMRT}), we have
 $\DD(A, \tau, h) \otimes_{K_0} K $ is Brauer equivalent to $A^{\otimes nm}$.  Suppose that 
 $\DD(A, \tau, h)$ is trivial. Then $A^{mn}$ is trivial. Since per$(A) = $ deg$(A) = 2n$, 
 $2n$ divides $nm$. In particular $2$ divides $m$. 
 \end{proof}
 
 If two central simple algebras $A$ and $B$ over a field $K$ are Brauer equivalent, then we write $A = B$. 
 
 For a cyclic extension $E/K$, $\sigma \in $ Gal$(E/K)$ a generator and $\mu \in K^*$,
 let  $(E,\sigma, \mu)$  be  the cyclic  algebra (\cite[Ch. V, \S 8]{Albert}).  We record here a well known property of 
 cyclic algebras. 
 
 \begin{lemma} 
 \label{cyclic} Let $K$ be a field and $E/K$ a finite cyclic extension. Let  $\sigma $ be a generator of 
 Gal$(E/K)$ and $\mu \in K^*$.  Let $K \subset L \subset E$ be a subfield. 
 Then \\
 i) $(E, \sigma, \mu) \otimes_K L  = (E/L, \sigma^{[L : K]}, \mu)$ \\
 ii) $(E, \sigma, \mu^{[E : L]})  = (L, \sigma\!\mid_{L}, \mu)$. 
 \end{lemma}
 
 \begin{proof} Let $M/K$ be any  field extension. Since $E/K$ is cyclic, 
 $E\otimes_K M \simeq \prod_i \tilde{E}$ for some cyclic extension $\tilde{E}/M$ and 
 the generator $\sigma$ of Gal$(E/K) $ induces a generator $\tilde{\sigma}$ of $\tilde{E}/M$.
 Further $(E, \sigma, \mu) \otimes_K M = (\tilde{E}, \tilde{\sigma}, \mu)$. 
 
 i) Suppose  $K \subseteq L \subseteq E$.   Since $E \otimes_K L \simeq \prod_i L$
 and $\sigma^{[L : K]}$ is the generator of Gal$(E/L)$ induced by $\sigma$,  we have  
 $(E, \sigma, \mu) \otimes_K L  = (E/L, \sigma^{[L : K]}, \mu)$.
 
 ii) This follows from (\cite[15, p.97]{Albert}). 
 \end{proof}
 
 \begin{prop}
 \label{quat} Let $K_0$ be a field of characteristic  not 2.  Let $K/K_0$ be a quadratic field extension. 
 Let $E_0/K_0$ be a   quadratic  field extension and $\sigma$
 the nontrvial automorphism of $E_0/K_0$.  Let $\mu \in K$. If cores$_{K/K_0}(E_0 \otimes K, \sigma \otimes 1, \mu)$ is trivial,  
 then there exists $\lambda  \in K_0$ such that $(E_0 \otimes K, \sigma \otimes 1, \mu) = (E_0, \sigma, \lambda) \otimes_{K_0}K$. 
 \end{prop}
 
 \begin{proof}  Suppose that cores$_{K/K_0}(E_0 \otimes K, \sigma, \mu)$ is trivial. 
 Since cores$_{K/K_0}(E_0 \otimes K, \sigma, \mu) = (E_0, \sigma, N_{K/K_0}(\mu))$,
 $N_{K/K_0}(\mu) \in N_{E_0/K_0}(E_0^*)$.  Let $\theta \in E_0$ be such that $N_{E_0/K_0}(\theta) =  N_{K/K_0}(\mu)$. 
Then, by (\cite[Lemma 2.13]{wads}), there exists $\lambda \in K_0$ such that 
$\lambda \mu \in N_{E_0 \otimes K/K}((E_0 \otimes K)^*)$. In particular  $(E_0 \otimes K, \sigma \otimes 1, \mu) =
 (E_0, \sigma, \lambda) \otimes_{K_0}K$.
 \end{proof}
 
  \begin{prop} 
  \label{odd-cores}
  Let $K_0$ be a field of characteristic  not 2.  Let $K/K_0$ be a quadratic field extension. Let $E/K$ be a cyclic extension of 
 odd  degree  and $\sigma$ a generator of Gal$(E/K)$.  Suppose that cores$_{K/K_0}(E,\sigma)$ is trivial. 
   Let $\mu \in K$. If cores$_{K/K_0}(E, \sigma, \mu)$ is trivial,  
 then there exists $\lambda  \in K_0$ such that $(E, \sigma, \mu) = (E, \sigma, \lambda) $. 
 \end{prop}
 
 \begin{proof} 
  Since 
 cores$_{K/K_0}(E,\sigma)$ is trivial, $(E,\sigma) \otimes_{K_0} K  \simeq  ( (E, \sigma), (E, \sigma^{-1}))$ over $K\otimes_{K_0} K \simeq K \times K$. 
Suppose cores$_{K/K_0}(E,\sigma, \mu) $ is trivial. Let $\bar{ }$ be the nontrivial automorphism of $K/K_0$. 
Then  $(E, \sigma, \mu) \otimes _{K_0} K \simeq ( (E, \sigma, \mu),   (E, \sigma^{-1}, \bar{\mu}))  = ( (E, \sigma, \mu),  (E, \sigma, \bar{\mu}^{-1}))$.
Since  cores$_{K/K_0}(E,\sigma, \mu) $ is trivial, 
$(E, \sigma, \mu) \otimes (E, \sigma, \bar{\mu}^{-1}) = 0$.    Thus 
$\mu \bar{\mu}^{-1} \in N_{E/K}(E^*)$   and hence $\mu^{-1} \bar{\mu} \in N_{E/K}(E^*)$. 

Let $n = [E : K]$ with $n$ odd.  If $n = 1$, then  $\lambda = 1$ has the required property. 
Suppose that $n \geq 2$. Then  $ \mu^n \mu^{-1} \bar{\mu} \in N_{E/K}(E^*)$.  Let $\theta = \mu \bar{\mu} \in K_0$. 
 Since  $\mu^n \mu^{-1} \bar{\mu} = \mu^{n-2} \mu \bar{\mu} = \mu^{n-2} \theta \in   N_{E/K}(E^*)$,
 $(E, \sigma, \mu^{n-2} ) = (E,\sigma,\theta^{-1})$.  
 Since $n$ is odd, $n-2$ is coprime to $n$ and hence there exists $1 \leq m \leq n-1$ such that 
 $(n-2)m \equiv 1$ modulo $n$.  We have 
 $$(E, \sigma, \theta^{-m}) = (E, \sigma, \theta^{-1})^m = (E, \sigma,  \mu^{n-2})^m = (E, \sigma,  \mu^{(n-2)m}) = (E, \sigma,  \mu).$$
 Hence $\lambda = \theta^{-m}$ has the required property.   
 \end{proof}
 
 \section{The Rost invariant} We recall   some facts on the  Rost invariant from (\cite[\S 31.B]{KMRT}). 
  \label{rost}
  
 Let $K$ be a field and $G$ a absolutely  simple  simply connected  linear algebraic group  over $K$. 
Rost defines an invariant,    
called the {\it Rost invariant},  
$$R_G : H^1(K, G) \to H^3(K, \Q/\Z(2)), $$ 
which is functorial for  field extensions of $K$.  
  
Let $G$ and $G'$ be two absolutely simple and simply connected groups over $K$ with a $K$-homomorphism
   $\phi : G \to G'$.  Then 
there exists a unique  positive integer $n_\phi$ such that  the following diagram 
$$
\begin{array}{cccc}
H^1(K, G)  & \buildrel{R_G}\over{\to} & H^3(K, \Q/\Z(2)) \cr
\downarrow \phi  & & \downarrow n_\phi \cr
H^1(K, G')  & \buildrel{R_{G'}}\over{\to} & H^3(K, \Q/\Z(2)) 
\end{array}
$$
commutes (\cite[p. 436]{KMRT}). In particular if $\phi$ is an isomorphism, then $n_\phi = 1$. 
In fact from the proof of the existence of such an  $n_\phi$ it follows that 
for any field extension $E/K$ we have  $R_{G'} \phi_E = n_\phi R_G$.

 Let $K$ be a field of characteristic not 2 and  $L/K$  a quadratic  extension. Let  $A$ be 
 a central simple algebra over $L$ with a
 $L/K$-involution $\tau$.  Let $(V,h)$ be a hermitian form over $(A,\tau)$. 
 Let $U(A,\tau, h)$ be the unitary group of $h$.  
 We have 
 $$U(A, \tau, h)(K)  = \{ f \in End_A(V) \mid h(f(v)) = h(v) {\rm ~for~ all~} v \in V \}.  $$ 
 Let $Nrd : End(V) \to L^*$ be the reduced norm map and 
 $L^{*1} = \{ a \in L^* \mid N_{L/K}(a) = 1 \}$. 
 Let $f \in U(A, \tau, h)(K)$. Then $Nrd(f) \in L^{*1}$.  
 The map $Nrd : End(V) \to L^ *$ induces a homomorphism 
 of algebraic groups $U(A, \tau, h) \to R^1_{L/K}(\G_m)$. 
 Let $SU(A, \tau, h)$ be the kernel of the homomorphism $U(A, \tau, h) \to R^1_{L/K}(\G_m)$.
  The we have 
  $$SU(A, \tau, h)(K)  = \{ f \in U(A, \tau, h) \mid Nrd(f) = 1  \}.$$
The  short  exact sequence of algebraic groups 
 $$ 1 \to SU(A, \tau, h) \to U(A, \tau, h) \to R^1_{L/K}\G_m \to 1 $$
  induces a  long exact sequence of cohomology sets   ($\star$) 
 $$  L^{*1} \buildrel{\delta}\over{\to} H^1(K, SU(A, \tau, h)) \to H^1(K, U(A, \tau, h)) \buildrel{d_h}\over{\to}
  H^1(K, R^1_{L/K}\G_m) = K^*/N_{L/K}(L^*).  $$
 The set $H^1(K, U(A, \tau, h))$ is in bijection with the set of isomorphism classes of hermitian forms over $(A, \tau)$ of 
 dimension equal to the dimension of $h$ and the map $ d_h 
 : H^1(K, U(A, \tau, h)) \to H^1(K, R^1_{L/K}\G_m) = K^*/N_{L/K}(L^*)$
corresponds to  the relative discriminant of hermitian forms, namely,   $d_h(h') = disc(h) disc(h')$ 
(\cite[Lemma 2.1.2]{BP}). 

Since $SU(A, \tau, h)$ is an  absolutely simple simply connected linear algebraic group, we have the Rost invariant map 
$$R_{SU(A, \tau, h)} : H^1(K, SU(A, \tau, h) \to H^3(K, \Q/\Z(2)). $$ 

Let $n = ind(A)$. Suppose that $2n$ is coprime to char$(K)$.
Then, using a result of Merkurjev-Suslin(\cite{MS}) on the divisibility of elements in $H^2(K, \Q/\Z(2))$, 
$H^3(K, \mu_{2n}^{\otimes 2})$ is identified with the subgroup of elements of 
exponent $2n$ in  $H^3(K,\Q/\Z(2))$  (cf. \cite[p. 436]{KMRT}). 
By (\cite[Proposition 31.43]{KMRT}),  the Rost invariant for $SU(A, \tau, h)$ takes values in the 
subgroup $H^3(K, \mu_{2n}^{\otimes 2})$ of $H^3(K,\Q/\Z(2))$. Hence we have 
$$R_{SU(A, \tau, h)} : H^1(K, SU(A, \tau, h) \to H^3(K, \mu_{2n}^{\otimes 2} ). $$

For any $\theta \in L^{*1}$,
 $R_{SU(A, \tau, h)} (\delta (\theta)) = cores_{L/K}(\mu \cdot A) $, where 
  $\theta = \mu \tau(\mu)^{-1}$ with $\mu \in L^{*1}$ (\cite[Appendix]{PP}).
 Let $h'$ be a hermitian form over $(A, \tau)$ with dim$(h') = $ dim$(h)$.
 Then $h'$ represents an element $H^1(K, U(A,\tau, h))$. Suppose that  $d_h(h') = 1$. Then there exists 
 $\zeta \in H^1(K, SU(A, \tau, h))$ which maps to $h'$ in $H^1(K, U(A,\tau, h))$. 
   It follows that 
the class of  $R_{SU(A, \tau, h)}(\zeta)$  in $H^3(K, \mu_{2n}^{\otimes 2} )/ cores_{L/K}(L^* \cdot A)$ is independent 
of the choice of the lift $\zeta$ of $h'$ (cf. \cite[p. 298]{preeti}) and  is
called the {\it Rost invariant of $h'$} relative to $h$ and denoted by $R_h(h')$. 

Let $h_0$ be the hyperbolic form over $(A, \tau)$ of dimension 2$m$. 
Let $h$ be a hermitian form over $(A, \tau)$ of dimension $2m$. Then 
$h$ represents an element in $H^1(K, U(A, \tau, h_0))$.
Suppose that $d_{h_0}(h) = 1$. Then $R_{h_0}(h)$ is denoted by $R(h)$ and is called 
the {\it Rost invariant} of $h$. 

Let $h$ and $h'$ be two hermitian forms over $(A,\tau)$ with 
dim$(h) = $ dim$(h')$. Suppose that $d_h(h') = 1$. 
Let $h_0$ be the hyperbolic form over $(A, \tau)$ of dimension 2dim$(h)$.
Since  $d_{h_0}(h \perp - h') = d_{h}(h') $ (\cite[Lemma 2.1.2]{BP}), we have $d_{h_0}(h \perp - h') = 1$.
Hence $R(h\perp -h')$ is defined.

 \section{Some computations of the Rost invariant}
 \label{rost-comp}
 Let $L = K(\sqrt{\theta})$ and   $h$  a hermitian form over $L/K$. Then $h = <a_1, \cdots, a_n>$ for some $a_i \in K^*$.
 Then the trace form $q_h$ of $h$ is given by $<1, -\theta><a_1, \cdots , a_n>$ is a quadratic form over $K$.
 The hermitian form $h$ is hyperbolic if and only is the trace form $q_h$ is hyperbolic over $K$ (\cite[Theorem 1.1]{Sch}).
Suppose  dim$(h)$ is even and disc$(h) = 1$. Then $ a = (-1)^{n(n-1)/2} \prod_i a_i \in N_{L/K}(L^*)$ (\cite[p. 350]{Sch}).
Since the Clifford algebra  $C(q_h)$ of $q_h$ is given by the quaternion algebra $(\theta,  a)$ (\cite[p. 350]{Sch}), 
$C(q_h)$ is trivial.

\begin{lemma} 
\label{arason}
 (\cite[p. 437]{KMRT})
 We have $R(h) = R(q_h)$ and $R(q_h)$ is the $e_3$ invariant of Arason .
\end{lemma}

 Suppose $A = M_n(D)$ for some central simple algebra $D$ over $L$ and $\tau$ a $L/K$-involution on $A$. 
 Then $D$ has a $L/K$-involution 
 $\tau_0$    and $\tau$ is the adjoint involution given by a hermitian 
 form $h_0$ over $(D, \tau_0)$ of dimension $n$. 
 Further the hermitian form $h$ on $(A, \tau)$ corresponds to an hermitian form $\tilde{h}$ over 
 $(D, \tau_0)$ by Morita equivalence.  
 Let $h'$ be a hermitian form over $(A, \tau)$ representing an element in $H^1(K, U(A, \tau, h))$.
Let   $\tilde{h}'$ be the  hermitian form over $(D,\tau_0)$ corresponding to $h'$ under  the  
 Morita equivalence. 
 
 \begin{lemma}
 \label{morita-rost-inv}   If $d_h(h') = 1$, then $d_{\tilde{h}}(\tilde{h}')   = 1$  and   $R_h(h') =  R_{\tilde{h}}(\tilde{h}')$.  
 \end{lemma}
 \begin{proof}
 We have canonical identification  of algebraic groups 
 $U(A, \tau, h) = U(D, \tau_0, \tilde{h})$ taking $SU(A, \tau, h)$ to $ SU(D, \tau_0, \tilde{h})$ and making the 
 following diagram commutative 
  $$ 
  \begin{array}{ccccccccc}
  1 & \to &  SU(A, \tau, h) &  \to & U(A, \tau, h)  & \to  & R^1_{L/K}\G_m & \to &  1 \cr 
  & & \downarrow & & \downarrow & & \downarrow id \cr 
  1 & \to &  SU(D, \tau_0, \tilde{h}) &  \to & U(D, \tau_0, \tilde{h})  & \to  & R^1_{L/K}\G_m & \to &  1 
  \end{array}
   $$
   with  vertical maps   isomorphisms.  We have the induced  commutative diagram  
   $$
     \begin{array}{ccccccccc}
  H^1(K, SU(A, \tau, h))  &  \to & H^1(K,  U(A, \tau, h) )  & \to  & H^1(K, R^1_{L/K}\G_m )   \cr 
   \downarrow & & \downarrow & & \downarrow id \cr 
  H^1(K,  SU(D, \tau_0, \tilde{h}))  &  \to & H^1(K,  U(D, \tau_0, \tilde{h}))  & \to  & H^1(K, R^1_{L/K}\G_m)  
  \end{array}
   $$
   with vertical maps isomorphisms. Since the image of $h'$ under the isomorphism $H^1(K, U(A,\tau, h)) \to H^1(K, U(D, \tau_0, \tilde{h})$ is $\tilde{h}'$, 
 it follows 
 that dim$(\tilde{h}') = $ dim$(\tilde{h})$,   $d_{\tilde{h}}(\tilde{h}') = d_h(h') = 1$  and   $R_h(h') =  R_{\tilde{h}}(\tilde{h}')$.  
\end{proof}

  The following is extracted from the proof of (\cite[Theorem 10.1]{PP}). 

 \begin{prop}
 \label{sumrost}  If dim$(h) = $ dim$(h')$ and $d_h(h') = 1$, then 
 $R(h \perp -h') =  R_h(h')$. 
 \end{prop} 
  
  \begin{proof} 
  Note that since $h \perp -h$ is hyperbolic, we have 
  $R(h \perp - h') = R_{h \perp -h}(h \perp -h')$. 
  Let $\zeta \in H^1(K, SU(A, \tau, h)$  which maps to $h'$ in $H^1(K, U(A, \tau, h))$.
Then, by definition, we have $R_h(h')  = R_{SU(A, \tau, h)}(\zeta) \in H^3(K, \mu_{2n}^{\otimes 2})/ cores_{L/K}(L^* \cdot A)$.

  Let $\eta : U(A, \tau, h) \to U(A, \tau, h \perp -h)$ be the homomorphism given by 
  $f \mapsto (f, 1)$.  Let $\eta^* : H^1(K, U(A, \tau, h)) \to H^1(K, U(A, \tau, h\perp -h))$ be the induced map. 
  Then the image of $h'$ under $\eta^*$ corresponds to $h' \perp -h$. 
  The map $\eta$ takes $SU(A, \tau, h)$ to $SU(A, \tau, h \perp -h)$ and hence induces a map
  $$\eta^* : : H^1(K, SU(A, \tau, h)) \to H^1(K, SU(A, \tau, h\perp -h)).$$  By (\cite[p. 436]{KMRT}), 
  there exists a positive integer $n_h$ such that $R(\eta^*(\zeta)) = n_hR(\zeta)$. 
  Since  $\zeta$ is a lift of $h'$, $\eta^*(\zeta)$ is a lift of $ \eta^*(h') = h' \perp -h$.

  Suppose $A$ is a split algebra.   
  Then, by Morita equivalence, $h$  and $h'$ corresponds to hermitian forms  $\tilde{h}$ and $\tilde{h}'$
  over $L/K$,  $R_{h}(h') = R_{\tilde{h}}(\tilde{h}')$ and $R(h \perp -h') = R(\tilde{h} \perp -\tilde{h}')$ (\ref{morita-rost-inv}). 
  Hence, replacing $A$ by $L$, $h$ and $h'$ by $\tilde{h}$ and $\tilde{h}'$,  we assume that $A = L$. 
  Let $q_h$ and $q_{h'}$ be the trace forms of $h$ and $h'$ respectively. By (\cite[31.42]{KMRT}), 
  we have $R(q_h \perp -q_{h'}) = R_{q_h}(q_{h'})$. Hence, by (\cite[Example 31.44 ]{KMRT}), $R(h \perp -h') = R_h(h')$

  Suppose $A$ is not split. 
  Let  $SB(A)$ be the Severi-Brauer variety of $A$ over $K$ and 
  $X = R_{L/K}(SB(A))$.
  Then $A\otimes_{K}K(X)$ is a split algebra (\cite[\S 4]{Scheiderer}).
  Hence, by the split case, we have $R(\eta^*(\zeta))_{K(X)}  = R(\zeta)_{K(X)}$.
  If necessary, going to a rational function field extension, 
  we assume that there are elements $\zeta \in H^1(K(X), SU(A, \tau, h))$
  with $R(\zeta)$ non trivial. 
  Hence  $n_h= 1$ and  $R(h \perp -h') =  R_h(h')$.
  \end{proof}
  
  \begin{prop}
 \label{rost-iso}   If dim$(h)$, dim$(h')$ are  even and $d(h) =  d(h') =  1$, 
 $R(h \perp h') =  R(h) + R(h')$.  In particular if  $h'$ is hyperbolic, then $R(h \perp h') = R(h)$. 
 \end{prop} 
  
  \begin{proof} 
  As in the proof of (\ref{sumrost}), we reduce it  to the Arason invariant of    quadratic forms over $F_0$.
  Since the Arason invariant $I^3(F_0) \to H^3(F_0,  \Q/\Z(2))$ is a homomorphism, we get that
  $R(h \perp h') = R(h) + R(h')$.
  \end{proof}

 Let  $A$ be a central simple algebra over $L$ with 
 a $L/K$-involution $\tau$.   Let  $\tau'$ be a  $L/K$-involution on $A$.
 Then there exists   unit $u \in A$ with $\tau(u) =  u$ such that $\tau' = int(u) \tau$ (\cite[Ch 8, Theorem 7.4]{Sch}). 
 Let $h$ be a hermitian form over $(A,\tau)$. Then $uh$ is a hermitian form over $(A, \tau')$. 
 Further we have an isomorphism of groups $U(A, \tau, h)$ and $U(A, \tau', uh)$ which takes 
 $SU(A, \tau, h)$ to $SU(A, \tau', uh)$.  Let $h'$ be a hermitian form over $(A, \tau)$ of dimension equal to dim$(h)$.
 Then $uh'$ is a hermitian   form over $(A,  \tau')$  which   corresponds to the image of $h'$ 
 under the induced map $H^1(K, U(A,\tau, h)) \to H^1(K, U(A, \tau, uh))$.
 We end this section with the following.
 \begin{lemma}
 \label{two-inv}  If  $d_h(h') = 1$, then $ d_{uh}(uh') = 1$ and   $R_h(h') = R_{uh}(uh')$.  
 \end{lemma}
 \begin{proof}
 We have the commutative diagram 
 $$ 
  \begin{array}{ccccccccc}
  1 & \to &  SU(A, \tau, h) &  \to & U(A, \tau, h)  & \to  & R^1_{L/K}\G_m & \to &  1 \cr 
  & & \downarrow & & \downarrow & & \downarrow id \cr 
  1 & \to &  SU(A, \tau', uh) &  \to & U(A, \tau', uh)  & \to  & R^1_{L/K}\G_m & \to &  1. 
  \end{array}
   $$
 Since the vertical maps in the above diagram are isomorphisms, it follows that 
  $d_h(h') = d_{uh}(uh')$ and   $R_h(h') = R_{uh}(uh')$. 
 \end{proof}

 Suppose that  $2(ind(A))$ is coprime to char$(K)$ and deg$(A){\rm dim}(h)$ is even.  
 For $\lambda \in K^*$, we  now compute the Rost invariant of $<1,\lambda>h$.
 
Let $  K(x)$ and $ L(x)$ be the rational function fields in one variable. 
 Let $h_0$ be the hyperbolic form over $(A, \tau)$ of dimension $2dim(h)$. 
 Then $<1, x>h $ represents an element in $H^1(K(x), U(A, \tau, h_0))$. 
  Since  deg$(A){\rm dim}(h)$ is even,  $d(<1, x>h) = 1$. Hence there exists $\xi \in H^1(K(x), SU(A, \tau, h_0)$
 which maps to the class of $<1, x>h $ in $H^1(K(x), U(A, \tau, h_0))$. 
 Let $n = ind(A)$. Then we have the residue map   $\partial_x : H^3(K(x), \mu_{2n}^{\otimes 2}) \to H^2(K, \mu_{2n})$ 
 at  $(x)$. 
 
 \begin{lemma} 
 \label{residue-ri}   Let $K$, $L$, $A$, $\tau$, $h$  and $\xi$ be as above. 
 Then $$\partial_x(R_{SU(A, \sigma, h_0)}(\xi)) = \DD(A, \tau, h).$$
 \end{lemma}
 
 \begin{proof}       Let  $SB(A)$ be the Severi-Brauer variety of $A$ over $L$ and 
  $Y = R_{L/K}(SB(A))$. Let $F_0$ be the function field of $Y$ over $K$
   and $F = L \otimes_K  F_0$.  
  Then $A\otimes_{L}F$ is a split algebra (\cite[\S 4]{Scheiderer}).
  Let $\tilde{h}$ be the hermitian form  over $F/F_0$ associated to $h_{F_0}$ by Morita 
  equivalence and $q_{\tilde{h}}$ the trace form of $\tilde{h}$.  Then the Clifford algebra $C( q_{\tilde{h}})$ of 
 $q_{\tilde{h}}$ is the quaternion algebra $(\theta, disc(h))$ (\cite[p. 350]{Sch}), where $L = K(\sqrt{\theta})$. 
 Since  $x \in F_0$,  $<1, x> \tilde{h}$ is the hermitian form over $F(x)/F_0(x)$ associated to $<1, x>h \otimes F_0(x)$ 
 under Morita equivalence and its 
 trace form is    $<1, x>q_{\tilde{h}}$.
 
 Since  $R(<1, x>h)_{F_0(x) }   = R(<1, x>\tilde{h}) $ (\ref{morita-rost-inv}) 
  and $R(<1, x>\tilde{h}) = R(<1, x> q_{\tilde{h}}) =  e_3(<1, x>q_{\tilde{h}})$ (\ref{arason}), 
 we have  $R(<1, x>h)_{F_0(x) }   = e_3(<1, x>q_{\tilde{h}}) =  ( (-x) \cdot C(q_{\tilde{h}}) )_{ F_0(x)} =  ((-x) \cdot (\theta, \disc(h)))_{F_0(x)}$.
  Since $\DD(A, \tau, h) \otimes_{K} F_0 = 
   (\theta, disc(h)) \otimes_{K}F_0$ (\cite[Proposition 10.33]{KMRT}), we have 
   $R(<1,x>h)_ {F_0(x)}  =   ( (-x) \cdot \DD(A, \tau, h))_ {  F_0(x)}$. 
   Since $A \otimes_{L} F$ is split, we have 
   $$R_{SU(A, \sigma, h_0)}(\xi_{F_0(x)}) = R(<1, x>h_{F_0(x)}) = 
   ((-x) \cdot \DD(A, \tau, h))_{F_0(x)}.$$   
   By taking the residue at $(x)$, we see that 
   $$\partial_x(R_{SU(A, \sigma, h_0)}(\xi)))_{F_0} = \partial_x(
   R_{SU(A, \sigma, h_0)}(\xi_{F_0(x)})) = \partial_x(  ((-x) \cdot \DD(A, \tau, h))_{F_0(x)} = \DD(A, \tau, h)_{F_0}.$$
Since the kernel of  the map $Br(K) \to Br(F_0)$  is generated by cores$_{L/K}(A)$ 
(\cite[Corollary 2.12 and Corollary 2.7]{MT}) 
and    cores$_{L/K}(A) = 0$,    the  map $Br(K) \to Br(F_0)$ is injective (cf. \cite[p. 132]{KMRT}). 
Hence $\partial_x(R_{SU(A, \sigma, h_0)}(\xi)) = \DD(A, \tau, h).$
 \end{proof}

 \begin{prop} 
 \label{ri-x} Let $K$ be a  field  with char$(K) \neq 2$. Let $L/K$ be a quadratic field extension. 
   Let $A$ be a central  division    algebra over $L$ with a 
 $L/K$-involution $\tau$. Let $h$ be a hermitian form over $(A, \tau)$. 
 Suppose that 2$(ind(A))$ is coprime to char$(K)$ and  deg$(A){\rm dim}(h)$ is even.  Let $  K(x)$ and $ L(x)$ be the rational function fields in one variable. 
 Then $d(<1, x>h) = 1$ and 
  $$R(<1, x>h) = \DD(A, \tau, h)  \cdot   (-x)  \in H^3(K(x), \mu_{2n}^{\otimes 2} ) / cores_{L(x)/K(x)}(L(x)^* \cdot A).$$
 \end{prop}
 
 \begin{proof}  Since  $x \in K(x)^*$ and deg$(A){\rm dim}(h)$ is even, $d(<1, x>h) =   1$.
  Hence the Rost invariant of $<1, x>h$ is defined. 
  
 Let $S = K[x, \frac{1}{x}] \subset K(x)$ and $T = L[x, \frac{1}{x}] \subset L(x)$. 
 The  exact sequence of algebraic group schemes 
 $$ 1 \to SU(A,  \tau,  h_0) \to U(A,  \tau,  h_0) \to R^1_{L/K}\G_m \to 1 $$
  induces a  long exact sequence of cohomology sets    
 $$ R^1_{L/K}\G_m(S)   \buildrel{\delta}\over{\to} H^1_{\acute et}(S, SU(A,  \tau,  h_0)) 
 \to H^1_{\acute et}(S, U( A,  \tau,  h_0))
  \buildrel{d_{ h_0}}\over{\to} H^1_{\acute et}(S, R^1_{L/K}\G_m) .  $$
 Since $S$ and $T$ are principal ideal domains,  we have 
 $R^1_{L/K}\G_m(S) = \{ a \tau(a)^{-1} \mid a \in T^* \}$ (cf. \cite[Proposition 2.3]{Suresh2024}) and 
 $H^1_{\acute et}(S, R^1_{L/K}\G_m) = S^*/N_{L/K}(T^*)$.

 Since  $x \in S^*$,
 $<1, x>h$ represents an element  $\zeta$ in $H^1_{\acute et}(S, U(A,  \tau,  h_0))$.
 Since   $d_{h_0}(<1, x>h) =   1$, there exists $\tilde{\xi} \in H^1_{\acute et}(S, SU(A,  \tau, h_0))$ which maps to $\zeta$. 
 Let $\xi \in H^1(K(x), SU(A, \tau, h_0)$ be the image of $\tilde{\xi}$.
 
 Let $n = ind(A)$ and    $\beta = R_{SU(A, \tau, h_0)}(\xi) \in H^3(K(x), \mu_{2n}^{\otimes 2})$. 
 Since $\xi$ is in the image of $H^1_{\acute et}(S, SU(A,  \tau, h_0))$, 
 $\beta$ is unramified at every maximal ideal of $S$ (\cite[Theorem 2]{Gille-rost-postive}).
  By (\ref{residue-ri}), we have  $ \partial_x(\beta)  = \DD(A, \tau, h)$. 
 Since $\DD(A, \tau, h)$ is defined over $K$,  $\DD(A, \tau, h) \cdot (-x)$ is unramified at every maximal ideal
  of $S$. In particular   $ \beta -   \DD(A, \tau, h) \cdot (-x)$  is unramified at every maximal ideal of $K[x]$. 
 Hence  there exists  $\beta_0 \in H^3(K, \mu_{2n}^{\otimes 2})$ with $\beta = \beta_{0K(x)} $ (cf.,  \cite[4.1]{CTSB}). 
 
  The homomorphism $S \to K$ given by $x = -1$ induces   the following commutative diagram  with exact rows
$$
  \begin{array}{ccccccc}
R^1_{L/K}\G_m(S)  &  \buildrel{\delta}\over{\to} &  H^1_{\acute et}(R, SU(A, \tau,  h_0)) & 
 \to &  H^1_{\acute et}(R, U( A, \tau,  h_0)) &   \buildrel{d_{ h_0}}\over{\to} &   H^1_{\acute et}(S, R^1_{L/K}\G_m)  \\ 
 \downarrow & & \downarrow & & \downarrow & & \downarrow  \\
R^1_{L/K}\G_m(K)  &  \buildrel{\delta}\over{\to} &  H^1(K, SU(A,  \tau,  h_0)) & 
 \to &  H^1(K, U(A,  \tau,  h_0)) &   \buildrel{d_{h_0}}\over{\to} &   H^1(K, R^1_{L/K}\G_m) 
  \end{array}
  $$
 
 Since $\zeta \in H^1_{\acute et}(R, U( A, \tau,  h_0))$ represents the hermitian form $<1, x>h$,
 $h$ is defined over $(A, \tau)$ and 
 the map $R \to K$ send $x$ to $-1$, the  image $\bar{\zeta}$  of $\zeta$ in $H^1(K, U(A, \tau, h_0)$ 
 represents $<1, -1>h$ and hence trivial.  
 Let $\bar{\xi}$ be the image of $\tilde{\xi}$ in $H^1(K, SU(A, \tau, h_0))$. Since 
 $\bar{\xi}$ maps to $\bar{\zeta}$, there exists $a \in L^*$ such that $\delta(a \tau(a)^{-1}) = \bar{\xi}$. 
Thus $R_{SU(A, \tau, h_0)}(\bar{\xi})) = cores_{L/K}( (a) \cdot A) $ (\cite[Appendix]{PP}). 

Let $\bar{\beta} \in H^3(K,\mu_{2n}^{\otimes 2})$ be the image of $\beta$ under the specialization map $x \to -1$.
By (\cite[Theorem 2]{Gille-rost-postive}), we have  $ \bar{\beta} = R_{SU(A, \tau, h_0)}(\bar{\xi}))$. Hence $\bar{\beta} = cores_{L/K} ( (a) \cdot A)$.  
Since $\bar{\beta} = \bar{\beta}_ {0K(x)} = \beta_0$, we have
 $\beta = \beta_{0K(x)} + \DD(A, \tau, h) \cdot (-x) = cores_{L/K} ( (a) \cdot A) _{K(x)} +  \DD(A, \tau, h) \cdot (-x)$.
Since  $\xi$ is the lift of $<1, x>h$, we have 
$$R(<1, x>h) =  \DD(A, \tau, h) \cdot (-x)  \in  H^3(K(x), \mu_{2n}^{\otimes 2} ) / cores_{L(x)/K(x)}(L(x)^* \cdot A).$$
 \end{proof}
 
\begin{theorem} 
\label{ri-lambda} Let $K$ be a  field  with char$(K) \neq 2$. Let $L/K$ be a quadratic field extension. 
   Let $A$ be a central  division    algebra over $L$ with a 
 $L/K$-involution $\tau$. Let $h$ be a hermitian form over $(A, \tau)$. 
 Suppose that 2$(ind(A))$ is coprime to char$(K)$ and  deg$(A){\rm dim}(h)$ is even.   Then for $\lambda \in K^*$,   $d(<1, \lambda>h) = 1$ and 
  $$R(<1, \lambda>h) = \DD(A, \tau, h)  \cdot  (-\lambda)  \in H^3(K, \mu_{2n}^{\otimes 2} ) / cores_{L/K}(L^* \cdot A).$$
\end{theorem}
  
  \begin{proof} Let $\lambda \in K^*$. Since deg$(A){\rm dim}(h)$ is even,     $d(<1, \lambda>h) = 1$.  

Let $h_0$ be the hyperbolic form over $(A, \tau)$ of dimension $2dim(h)$. Then $<1, x>h$ represents an element in $H^1(K(x), U(A, \tau, h_0))$
and it lifts to an element   $\xi \in H^3(K(x), SU(A, \tau, h_0))$. Then, from the proof of (\ref{ri-x}), 
we have $R_{SU(A, \tau, h_0)}(\xi) = cores_{L/K}((a) \cdot A) +  \DD(A, \tau, h) \cdot (-x) $ for some $a \in L^*$.

By specializing at $x = \lambda$, we see that, $R_{SU(A,\tau, h_0)}(\xi)$ maps to 
$cores_{L/K}((a) \cdot A) +  \DD(A, \tau, h) \cdot (-\lambda). $ Since the Rost invariant commutes with the specialization 
(\cite[Theorem 2]{Gille-rost-postive}), 
we have $R_{SU(A, \tau, h_0)}(\bar{\xi}) = cores_{L/K}((a) \cdot A) +  \DD(A, \tau, h) \cdot (-\lambda). $
Since $\bar{\xi}$ lifts the class of the hermitian form $<1, \lambda>h$, we have 
$$R(<1, \lambda>h) = \DD(A, \tau, h)  \cdot  (-\lambda)  \in H^3(K, \mu_{2n}^{\otimes 2} ) / cores_{L/K}(L^* \cdot A).$$

  \end{proof}

\section{ Some  Classification  results for  hermitian forms}
\label{class-gen}

\begin{prop}
\label{dim3} Suppose that $L= K(\sqrt{\theta})$  is a quadratic extension and  $D_0$   a quaternion division algebra 
over $K$.  Let  $\star$  be   the canonical   involution on $D_0$, $\iota$  the nontrivial automorphism of $L/K$,
$D = D_0 \otimes_K L$  and 
$\tau = \star \otimes \iota$. 
Let $h_1$ be a hermitian form over $(D_0, \star)$ and $h_2$ a skew-hermitian form over $(D_0, \star)$. 
 Then   $h = h_1 \perp \sqrt{\theta} h_2$  is a hermitian form over $(D, \tau)$. 
Suppose that dim$(h)$ is even, disc$(h)$ is trivial and $R(h)$ is trivial. 
 If  dim$(h_2) \leq 3$, then $<1, -\theta>h_2$ is hyperbolic as skew-hermitian form over $(D_0, \star)$.  
 \end{prop}

\begin{proof}   
 Let $C$ be the conic associated to $D_0$ and  $K(C)$ be the function field of $C$ over $K$. 
Since $D_0 \otimes K(C)$ is split (\cite[Proposition 7]{roquette}), $h_1 \otimes K(C)$ 
 corresponds to a skew-symmetric form over $K(C)$ and hence 
 hyperbolic. 
In particular  dim$(\sqrt{\theta} h_2 \otimes K(C))$ is even, disc$(\sqrt{\theta} h_2 \otimes K(C) )$ is trivial 
and $R(\sqrt{\theta} h_2 \otimes K(C))$ is trivial.

Since $D_0 \otimes K(C) \simeq M_2(K(C))$, by Morita equivalence,  
the skew-hermitian form $h_2  \otimes K(C)$ over $(D_0\otimes K(C), \star \otimes 1)$
corresponds a quadratic form $<a_1, \cdots , a_{2n}>$ over $K(C)$. Further 
the hermitian form $ (\sqrt{\theta} h_2) \otimes K(C) $ over $(D\otimes K(C), \tau \otimes 1)$
corresponds to the hermitian form $\tilde{h}_2 = <a_1, \cdots ,  a_{2n}>$  over $L(C)/K(C)$.   
 Since Morita equivalence preserves the discriminant and Rost invariant, 
disc$(\tilde{h}_2)$ is trivial and $R(\tilde{h}_2)$ is trivial.  
Let $q_{\tilde{h}_2}  = <1, -\theta><a_1, \cdots , a_{2n}>$ 
be the trace form of $\tilde{h}_2$.   Then dim$(q_{\tilde{h}_2})$ is even and disc$(q_{\tilde{h}_2})$ is trivial.
Further  $C((q_{\tilde{h}_2}))$ is trivial  (\cite[p. 350]{Sch})). 
The Rost invariant of $\tilde{h}_2$ is the Arason invariant of $q_{\tilde{h}_2}$  (cf. \ref{arason}) which is trivial. 
Hence $q_{\tilde{h}_2} \in I^4(K(C))$ (\cite{ovv}). 

Suppose dim$(h_2) \leq 3$. Then dim$(\tilde{h}_2) \leq 6$ and hence  dim$(q_{\tilde{h}_2}) \leq 12$.
Since, by a result of Arason and Pfister (cf. \cite[Ch. 7, Remark 7.4]{Sch}),  
every  anisotropic quadratic form representing an element  $I^4(K(C))$ has dimension at least 16, 
 $q_{\tilde{h_2}} = <1, -\theta><a_1, \cdots , a_{2n}>$ is hyperbolic.   
 
 Since  the quadratic form $<a_1, \cdots , a_{2n}>$ over $K(C)$  corresponds to the skew hermitian form 
 $h_2$ over $(D_0 \otimes K(C), \star)$ under the Morita equivalence, the quadratic form $q_{\tilde{h}_2}  = <1, -\theta><a_1, \cdots , a_{2n}>$  
 corresponds to  the skew hermitian form $<1, -\theta>h_2$  over $(D_0 \otimes K(C), \star)$ under the Morita equivalence.
 Since $q_{\tilde{h_2}}$ is hyperbolic, 
  $<1, -\theta>h_2$ is in the kernel of the homomorphism $W^{-1}(D_0, \star) \to W^{-1}(D_0\otimes K(C), \star)$ which is   injective (\cite[Proposition 3.3]{PSS}).
  Thus 
  $<1, -\theta>h_2$ 
is hyperbolic as a skew-hermitian form over $(D_0, \star)$.  
\end{proof}

 We  recall   the following well known fact (cf. \cite[Proof of Theorem 4.3]{preeti}). 

\begin{prop} 
\label{odd}  Let $F_0$ be a field with char$(F_0) \neq 2$  and    cd$_2(F_0) \leq 3$.
Let $F/F_0$ be a quadratic extension  and $D/F$ a central division  algebra with a $F/F_0$-involution $\tau$.  
Suppose that deg$(D)$ is odd. 
  Let $h$   be a  hermitian form  over $(D, \tau)$. If dim$(h)$ is even, disc$(h)$ is trivial and 
 Rost invariant $R(h)$  of $h$ is trivial, then $h$ is hyperbolic. 
\end{prop}

\begin{proof} Since deg$(D)$ is odd, there exists an extension $L_0/F_0$ of odd degree such that 
$D\otimes_{F_0} L_0$ is split (cf. \cite[Lemma 3.3.1]{BP}). Let $L = F \otimes L_0$.  By Morita equivalence, 
$h$ corresponds to a hermitian form $\tilde{h}$ over $L/L_0$. 
Let $q_{\tilde{h}}$ be the trace form of $h$.  Then, as in the proof of (\ref{dim3}), 
$q_{\tilde{h}} \in I^4(L_0)$.  Since cd$_2(F_0) \leq 3$, cd$_2(L_0) \leq 3$ (\cite[p. 83, Proposition 10]{SerreGC}).
Hence $I^4(L_0) = 0$ (\cite[Corollary]{AEJ}). In particular $q_{\tilde{h}}$ is hyperbolic over $L_0$
 and hence $h$ is hyperbolic over $L_0$. Since  $[L_0 : F_0]$ is odd, by (\cite{BL}), $h$ is hyperbolic.  
\end{proof}

 \section{Global fields and local fields}
 \label{global-local}
 
 For a non archimedean local field $K$, we have a canonical  isomorphism $inv: Br(K) \to \Q/\Z$
 and $inv : Br(\R) \to \Z/2\Z \subset \Q/\Z$ (\cite[p. 130]{CFANT}).
 For a global field  $K$, let $\Omega_K$ is the set of places of $K$. 
 For $\nu \in \Omega_K$, let $K_\nu$ denote the completion of $L$ at $\nu$ and 
 $inv_\nu : Br(K_\nu ) \to \Q/\Z$ be  the canonical map. 
For $A \in Br(K)$ and $\nu \in \Omega_K$, let $inv_\nu(A) = inv_\nu(A\otimes K_\nu)$.
Then $A \in Br(K)$, $\sum_\nu inv_\nu(A) = 0$ for all $A \in Br(K)$and
$A$ is a split algebra if and only if $inv_v(A) =0$ for all $v \in \Omega_K$  ( (\cite[p. 188]{CFANT}) . 
Further  given $(A_\nu) \in \oplus Br(K_\nu)$ with $\sum_\nu inv_\nu(A_\nu)  = 0$,   there exists 
 $A \in Br(K)$ with $inv_\nu(A) = inv_\nu(A_\nu)$  (\cite[p. 196 ]{CFANT}).

 \begin{prop}
 \label{dalgebra} Let  $K_0$ be  a global field  with no real places or a local field. Suppose that 
 char$(K_0) \neq 2$.
 Let  $K = K_0(\sqrt{a})$ be a quadratic extension.   
 Let $A$ be a central division algebra  over $K$ 
  with a $K/K_0$-involution $\tau$. Let $h$ be a hermitian form over $(A, \tau)$ with dim$(h)$deg$(A)$ even. 
   If the discriminant algebra  $\DD(A, \tau, h)$ of $(A, \tau, h)$ is trivial, then 
 $h$ is  hyperbolic. 
 \end{prop}
 
 \begin{proof}   Suppose that the discriminant algebra $\DD(A, \tau, h)$ of $(A, \tau, h)$ is trivial. 
 Since $K$ is a global field or a local field, per$(A) =$ deg$(A)$.
 If  deg$(A)$ is even,   then  by (\ref{even}), dim$(h)$ is even.
 If deg$(A)$ is odd, since  dim$(h)$deg$(A)$ is even,  dim$(h)$ is even.
  
 Suppose $K_0$ is a local field. Then $A = K$ (\cite[Ch.10, Theorem 2.2]{Sch}) and 
  by (\cite[Proposition 10.33]{KMRT}),
 $\DD(A, \tau, h)$ is isomorphic to the quaternion algebra  $(a, disc(h))$. 
 Since $\DD(A, \tau)$ is trivial, disc$(h)$ is a norm from the 
extension $K/ K_0$. In particular  disc$(h)$ is trivial and hence
$h$ is hyperbolic (\cite[Ch.10, Examples 1.6(ii)]{Sch}). 
 
Suppose $K_0$ is a global field with no real places. 
Let $\nu$ be a place of $K_0$. If $K\otimes K_{0\nu}$ is not a field, then $h$ is hyperbolic over 
 $K_{0\nu}$ (\cite[Ch.10, Remark 6.3]{Sch}).  Suppose $K\otimes K_{0\nu}$ is  a field.
 Then $A \otimes K_{0\nu} \simeq M_{n}(K\otimes K_{0\nu})$ (\cite[Ch.10, Theorem 2.2]{Sch}). 
 Hence, by Morita  equivalence, $h$ corresponds to  a hermitian form  $h_\nu$  over
  $K\otimes K_{0\nu}/K_{0\nu}$.  Since   $\DD(A, \tau, h) \otimes K_{0\nu}$ is isomorphic to the quaternion algebra 
$(a, disc(h_\nu))$,  disc$(h_\nu)$ is  trivial. Hence, as above,  
$h_\nu$ is hyperbolic. 
 Since locally hyperbolic forms are hyperbolic  (\cite[Ch.10, Theorem 6.1]{Sch}), $h$ is hyperbolic.   
 \end{proof}

  \begin{prop}
\label{induction-step} Suppose $K_0$ is a global field. 
 Let $K/K_0$ be a quadratic extension and  $K_1/K_0$  a quadratic extension with $E_1 = K\otimes_{K_0}K_1$ a field. 
 Let $A$ be a central simple algebra over $K$ with cores$_{K/K_0}(A) = 0$. Then there exists a central simple algebra 
 $B$ over $E_1$ with cores$_{E_1/K}(B) = A$ and cores$_{E_1/K_1}(B) = 0$. Further if $E/E_1$ is a finite extension 
 with $A\otimes_KE$ splits,  then we can choose $B$ with $B\otimes_{E_1} E$ splits.  
 \end{prop}
 
 \begin{proof}  Let $S$ be the finite set of places  $v$ of $K_0$ such that $A  \otimes_K K_w$ is nontrivial for some 
 place $w$ of $K$ lying over $v$. 
 
  Let $v \in S$. Since cores$_{K/K_0}(A) = 0$ and 
 $A  \otimes K_w \neq 0$
 for some valuation $w$ of $K$ lying over $v$, $v$ splits in $K$ (\cite[Ch.10, Theorem 2.2]{Sch}). 
 Let $v_1, v_2$ be the places of $K$ lying over $v$.  Then $inv_{v_1}(A) + inv_{v_2}(A) = 0$. 
 
 Suppose $v_1$   splits in $E_1$.  Since $E/K_0$ is a biquadratic   extension,
 $v_2$  splits in $E_1$.  Let $v_{i1}$ and $v_{i2}$ be the places of $E_1$
 lying over $v_{i}$ for $i = 1, 2$. Further $v$  split in $K_1$. Let $w_1$ and $w_2$ be the places of $K_1$ lying over 
 $v$.  Then $w_1$ and $w_2$ split in $E_1$. By reindexing, we assume that 
 $v_{11}$ and $v_{22}$ lye over $w_1$ and $v_{12}$ and $v_{21}$ lye over $w_2$. 
 
 Suppose $v_1$ is inert in $E_1$.  Since $E_1/K_0$ is Galois, $v_2$ is also inert in $E_1$. 
 Let $\tilde{v}_i$ be the unique place of $E_1$  lying over $v_i$ for $i = 1, 2$. Further 
 $v$ is inert in $K_1$. Let $\tilde{v}$ be the unique place of $K_1$ lying over $v$. 
 Then $\tilde{v}$ splits in $E_1$ and $\tilde{v}_1$, $\tilde{v}_2$ are the paces of 
 $E_1$ lying over $\tilde{v}$.

 Let $w$ be a place of $E_1$. We define $n_w \in \Q/\Z$ as follows. 
 Let  $v$  be the restriction of $w$ to $K_0$. 
  Suppose $v \not\in S$. Let $n_w = 0$.  Suppose $v \in S$.  Then $v$ is not inert in $E_1$. 
  Hence either $v$ splits into 4 places $v_{ij}$ of $E_1$ or 
  $v$ splits into 2 places $\tilde{v}_i$ of $E_1$ as above. 
  
 Suppose  $v$ splits into 4 places of $E_1$. Then $w=  v_{ij}$ for some $1 \leq i, j \leq 2$.
 If $w = v_{ii}$,  let $n_w = $ inv$_{v_i}(A)$. If $w = v_{ij}$ with $i \neq j$, let $n_w =  0$. 
 
 Suppose $v$ splits into two places of $E_1$. Then $w = \tilde{v}_i$ for some $i = 1, 2$.
 Let $n_w = $ inv$_{v_i}(A)$. 
 
 First note that by the choice of $n_w$,  for any place $v$ of $K_0$, we have $\sum_{w \mid v} n_w = 0$
 and $n_w \neq 0 $ only for finitely many places $w$ of $E_1$. Hence there exists a central simple algebra 
 $B$ over $E_1$ with  inv$_w(B) = n_w$ for all places $w$ of $E_1$.   
  
  We show that $B$ has the required properties. 
  Let $w$ be a place of $K$ and $v$ its restriction to $K_0$.
  Suppose that $v \not\in S$. Then inv$_w(A) = 0$ and inv$_{\tilde{w}}(B) = 0$ for all  places $\tilde{w}$ of $E_1$ lying over 
  $w$. In particular inv$_w(cores_{E_1/K}(B))  = 0 = inv_w(A) $.  Suppose that $v \in S$.
  Then $w = v_i$ for some $i = 1, 2$, say $i = 1$.  Suppose $v_1$ is inert in $E_1$. Let $\tilde{v}_1$ be the unique place of 
  $E_1$ lying over $v_1$. Then by the choice of $B$
  we have inv$_{\tilde{v}_1}(B) = inv_{v_1}(A)$.  Since $E_1 \otimes_K K_{v_1}$ is a field, 
  cores$_{E_1\otimes_K K_{v_1}/K_{v_1}}$ is  the identity map
  (\cite[Theorem 10, p. 237]{Lorenz}), we have $ inv_{v_1}(A) = inv_{\tilde{v}_1}(B) =    inv_{v_1} (cores_{E_1 / K}(B))$. 
  Suppose that $v_1$  splits in $E_1$. Then we have 2 places $v_{11}$ and $v_{12}$ of $E_1$ lying over $v_1$.
By the choice of $B$ we have  inv$_{v_{11}}(B) = inv_{{v_1}}(A)$ and $inv_{v_{12}}(B) = 0$.
We have  inv$_{v_1}(cores_{E_1/K}(B))  = inv_{v_{11}}(B)+ inv_{v_{12}}(B) = inv_{v_1}(A)$.  
Hence  inv$_{w}(cores_{E_1/K}(B))  =   = inv_{w}(A)$ for all places $w$ of $K$. Thus  $cores_{E_1/K}(B))  =  A$. 
 
Let $w_1$ be a place of $K_1$ and $v$ its restriction to $K_0$. Suppose $v \not \in S$.  Then,  
as above, inv$_{w_1}(cores_{E_1/K}(B))  = 0$. 
Suppose $v \in S$.  Let $v_1$ and $v_2$ be the places of $K$ lying over $v$.
Suppose $v$ splits into two places  $w_1$ and $w_2$ of $K_1$. Then $v$ splits in to four places $v_{ij}$ of $E_1$ as above
and $v_{11}$, $v_{22}$ lye over $w_1$  or   $v_{12}, \nu_{21}$ lye over $w_1$.
If  $v_{11}$, $v_{22}$ lye over $w_1$, 
then   $inv_{w_1}cores_{E_1/K_1}(B) = inv_{v_{11}}(B) + inv_{v_{22}}(B) = inv_{v_1}(A) + inv_{v_2}(A) = 0$.
If  $v_{12}$, $v_{21}$ lye over $w_1$, then   $inv_{w_1}cores_{E_1/K_1}(B) =  0+ 0 = 0$. Hence cores$_{E_1/K_1}(B) = 0$. 

Suppose $E/E_1$ is a finite extension with $A \otimes_K E$ splits.   Let $\tilde{w}$ be a place of $E$,  
 $w$ the restriction $\tilde{w}$  to  $E_1$ and 
 $v_1$ the restriction of $\tilde{w}$ to $K$.   Then by the choice of $B$, we have $inv_{w}(B)$ is 
either 0 or equal to $inv_{v_1}(A)$. 
  Since $A\otimes_KE$ splits, it follows that $inv_{\tilde{w}}(B\otimes E)$ is zero. Hence 
  $B\otimes_{E_1} E$ splits. 
   \end{proof}
 
  \begin{prop}
\label{globalfield}
 Let $K/K_0$ be a quadratic extension. Let $E/K$ be a cyclic extension such that 
 $E/K_0$ a dihedral extension.
  Let  $\sigma$ be a generator of Gal$(E/K)$ and $\mu \in K^*$. If  cores$_{K/K_0}(E, \sigma,  \mu) = 0$,
 then there exists $\lambda \in K_0$ such that $ (E, \sigma, \mu) = (E, \sigma, \lambda)$.
 \end{prop}
 
 \begin{proof}   
 Suppose $K_0$ is a local field. Suppose cores$_{K/K_0}(E, \sigma,  \mu) = 0$.
 Then $(E, \sigma, \mu) = 0$ ((\cite[Ch.8, Theorem 9.5 and Ch.10, Theorem 2.2]{Sch})). 
 Hence $\lambda = 1$ has the required property. 
 
 Suppose that $K_0$ is a global field with no real places.  Write $d = [E : K] = 2^n m$ with $n \geq 0$ and $m$ odd.
 We prove the result by induction on $n$. If $n = 0$, then $[E : K]$ is of odd degree and hence
  by  (\ref{odd-cores}),
 there exists $\lambda \in K_0$ such that $ (E, \sigma, \mu) = (E, \sigma, \lambda)$.
 
 Suppose that $n \geq 1$.  Then  $[ E : K] $ is even.     
 Let $E_1$ be the unique sub extension of  $E/K$ with $[E_1 : K] = 2$.
 Since   $E_1/K_0$ is a dihedral extension of degree 4,  the Galois group of $E_1/K_0$ is non-cyclic  abelian group of 
 order 4. Hence there exists a sub extension $K_1/K_0$ of $E_1/K_0$ with $[K_1 : K_0] = 2$ and 
 $K_1K = E_1$.  
 
 Let $A = (E, \sigma, \mu)$.  Since cores$_{K/K_0}(A) = 0$, by (\ref{induction-step}), there exists 
 a central simple algebra $B$ over $E_1$ such that $cores_{E_1/K}(B) = A$ and $cores_{E_1/K_1}(B) = 0$. 
 Since $A \otimes_K E$ is split, further,  by (\ref{induction-step}), we choose $B$ such that $B\otimes_{E_1}E$ is split. 
 Since $E/E_1$ is cyclic with $\sigma^2$ a generator of Gal$(E/E_1)$, 
 there exists $\mu_1 \in E_1$ such that $B = (E, \sigma^2, \mu_1)$ (\cite[Theorem 3, p. 94]{Albert}).
 Since $[E : E_1] < [ E : K]$, by  induction hypothesis, there exists $\lambda_1 \in K_1$ such that 
 $B  = (E/E_1, \sigma^2, \lambda_1)$.  We have 
 $A  =  cores_{E_1/K}(B) = (E, \sigma,  N_{E_1/K}(\lambda_1))$.   Let $\lambda  = N_{E_1/K}(\lambda_1)$.
 Then  $(E, \sigma, \mu) = A  = (E, \sigma, \lambda)$. 
 \end{proof}

 \section{Morita Equivalence} 
 \label{Morita}
 
 We refer  to (\cite{Knus}) for the generalities on Morita equivalences. 
 Let $F$ be a field and $D/F$ a central simple algebra.  Let $E/F$ be a quadratic separable extension. Suppose that 
 $E\subset D$. Let $D_1$ be the commutant of $E$ in $D$. Then we have an isomorphism of $E$-algebras
 $$ \psi : D \otimes_F E \to End_{D_1}(D),$$
 where $D$ is considered as the right  $D_1$-module with the action $ \lambda \in D_1, x \in D$, $ x \cdot   \lambda      =  x  \lambda $. 
 We  have 
 $\psi(a  \otimes \alpha )(x) = a x \alpha $ for all $a, x \in D$ and $\alpha \in E$.  

 The non trivial automorphism of $E/F$ extends to an automorphism of $D$ given by int$(z)$ for some $z \in D$. 
Since $zEz^{-1} = E$ and int$z^2$ restricted to $E$ is identity, we have $zD_1z^{-1} = D_1$  and $z^2 \in D_1$. 
We have $D = D_1 \oplus zD_1$.  Let $w = z^2 \in D_1$. 
 Then we fix the basis $\{ 1, z  \}$ and fix the isomorphism $\psi : 
 D\otimes_FE \to M_2(D_1)$ with respect to this basis. 
 Let   $$ W = 
 \begin{pmatrix}
 0   &  w \\
 1 & 0 
 \end{pmatrix} \in M_2(D_1)
 $$
 Then $\psi(z \otimes 1) = W$. 
 
 Let $\phi = int(z)$.  Then $\phi(D_1) = D_1$ and $\phi(E) = E$. 
  Let $a \in D_1$ and $\alpha \in E$. Then 
  $$\psi(a \otimes \alpha) =  
 \begin{pmatrix}
a  \alpha  & 0   \\
  0 & \alpha \phi^{-1}(  a)  
 \end{pmatrix}. 
  $$
 
Let $F_0 \subseteq F$ with $[F : F_0] \leq 2$.  Suppose there is a $F/F_0$- involution 
$\tau$ such that $\tau(D_1) = D_1$. Then $\tau(E) = E$. Suppose that  $\tau(z) = z$. 
 Let $\tau_1$ be the restriction of $\tau$ to $D_1$. 
  Let $\iota$ be the  restriction of $\tau$ to $E$. 

 Let $\tilde{\tau}  = \psi (\tau \otimes \iota) \psi^{-1}$ and  $\tilde{\tau}_1$ be the involution on $M_2(D_1)$ given by 
$\tilde{\tau}_1( (a_{ij}) ) = (\tau_1(a_{ji}))$ for all $(a_{ij}) \in M_2(D_1)$. 

  Let  $H =    \begin{pmatrix}
  1 &  0 \\
  0  &  w^{-1}  
 \end{pmatrix} \in M_2(D_1)$. 
 
 \begin{prop}
 \label{morita}
 $\tilde{\tau} = int(H) \tilde{\tau}_1$. 
 \end{prop}

\begin{proof}
Let   $a \in D_1$. Then 
$$\tilde{\tau}(\psi(a \otimes 1)) =  \psi(  \tau \otimes \iota )(a \otimes 1) =  \psi (\tau_1(a) \otimes 1)  = 
 \begin{pmatrix}
 \tau_1(a ) &    0 \\
 0 &  \phi^{-1}( \tau_1(a))  
 \end{pmatrix} $$ 
 and
 $$\tilde{\tau}_1(\psi(a\otimes 1)) =  \tilde{\tau}_1(  \begin{pmatrix}
 a  &    0 \\
 0 &    \phi^{-1}a    
 \end{pmatrix} ) =   \begin{pmatrix}
 \tau_1(a) &    0 \\
 0 &  \tau_1( \phi^{-1}a ) 
 \end{pmatrix}.  $$

We have
 
 $$
 \begin{array}{rcl}
 H \tilde{\tau}_1(\psi(a \otimes 1))  H^{-1}  &   =  &   \begin{pmatrix}
 1   &    0   \\
0   &  w^{-1} 
 \end{pmatrix}
   \begin{pmatrix}
 \tau_1(a) &    0 \\
 0 &  \tau_1( \phi^{-1} a)  
 \end{pmatrix}  \begin{pmatrix}
 1   &    0   \\
0   &  w 
 \end{pmatrix}  \\ \\
  & = & 
 \begin{pmatrix}
 \tau_1(a)    &  0   \\
 0   &   w^{-1}\tau_1( \phi^{-1}a)  w 
 \end{pmatrix} 
 \end{array}. 
 $$
 
 Since $w^{-1} \tau_1( \phi^{-1}a)  w  = w^{-1} z \tau_1(a) z^{-1} w = z^{-1} \tau_1(a)  z = \phi^{-1} \tau_1(a)$, 
 we have $$\tilde{\tau}_1(\psi(a\otimes 1)) = H \tilde{\tau}_1(\psi(a \otimes 1))  H^{-1}$$ for all $a \in D_1$.

We have $\tilde{\tau}(\psi(z \otimes 1)) = W$ and 
$$ \begin{array}{rcl}H \tilde{\tau}_1(\psi(z \otimes 1))H^{-1} &  =  &  H W^t H^{-1} \\
& = & \begin{pmatrix}
 1   &    0   \\
0   &  w^{-1} 
 \end{pmatrix}
\begin{pmatrix}
 0 &    1   \\
w   &  0  
 \end{pmatrix}
\begin{pmatrix}
 1   &    0   \\
0   &  w 
 \end{pmatrix} \\ \\
 & = & 
\begin{pmatrix}
 0 &    1   \\
1  &  0  
 \end{pmatrix}
\begin{pmatrix}
 1   &    0   \\
0   &  w 
 \end{pmatrix} \\ \\
 & = & 
\begin{pmatrix}
 0 &    w   \\
1  &  0  
 \end{pmatrix}
 = W  = \tilde{\tau}(\psi(z \otimes 1))
\end{array}. 
$$
     
Since $D_1$ and $z$ generate $D$,  $\psi(D_1 \otimes 1)$  and  $\psi(z \otimes 1)$ generate $M_2(D_1)$.
Hence $\tilde{\tau} = int(H) \tilde{\tau}_1$. 
\end{proof}

\begin{cor}
\label{morita-cor} Let $a \in D_1$ with $\tau_1(a )  = a$.   
 Then  the image of the  hermitian form $<a \otimes 1>$  over $D \otimes_F E$  under $\psi$ corresponds to 
  the hermitian form 
 $\begin{pmatrix}
a & 0   \\
  0 &  \phi^{-1}(  a)  
 \end{pmatrix}$ over $M_2(D_1)$  and it corresponds to the hermitian form 
 $\begin{pmatrix}
 a & 0   \\
  0 &  w^{-1} \phi^{-1}(  a)  
 \end{pmatrix}$  over $D_1$ 
 under   the Morita equivalence. 
 \end{cor}
 
 \begin{proof}  Let $a  \in D_1$. Since  $\psi(a  \otimes 1) = \begin{pmatrix}
a & 0   \\
  0 &  \phi^{-1}(  a)  
 \end{pmatrix}$, the first assertion follows.  Since $\tilde{\tau} = int(H) \tilde{\tau}_1$ (\ref{morita}), 
 under the Morita equivalence, 
 the hermitian form $\begin{pmatrix}
a & 0   \\
  0 &  \phi^{-1}(  a)  
 \end{pmatrix}$ corresponds to the hermitian form 
 $H \begin{pmatrix}
a & 0   \\
  0 &  \phi^{-1}(  a)  
 \end{pmatrix} = \begin{pmatrix}
 a & 0   \\
  0 &  w^{-1} \phi^{-1}(  a)  
 \end{pmatrix}$. 
 \end{proof}

\section{Classification of Hermitian forms}
\label{classification}

Let $F$ be  a complete discretely valued field with valuation ring $R$ and  residue field  $K$   of characteristic not 2.   
 Let $D$ be a central division algebra over $F$.  
 Let $\nu$ be the valuation on $F$. 
  Then   $\nu$  extends to a unique  valuation  $\nu_D$ on $D$ ((\cite[Theorem 12.10, p. 138]{reiner})).
 Further $\Lambda = \{ a \in D \mid \nu_D(a) \geq 0 \}$ is the unique maximal $R$-order of $D$
  (\cite[Theorem 12.8]{reiner}).
 Let $\pi_D \in m_D$ be a parameter; i.e  $\pi_D \in m_D$ with $\nu_D(\pi_D) = 1$. 
  Then $m_D =  \Lambda \pi_D = \pi_D \Lambda =  \{ a \in D \mid \nu_D(a) > 0 \}$
  is the unique 2-sided maximal ideal of $\Lambda$  and  $\bar{D} = \Lambda/m_D$  is a division algebra with
   center a finite extension of  the residue field $K$ of $F$ (\cite[Theorem 13.2 and Theorem 13.3]{reiner}).  
   
 Suppose that  per($D$) is coprime to char$(K)$.  Let $\pi \in F$ be a parameter. 
We have 
$D = D_0 + (E, \sigma, \pi) \in Br(F)$ for some  unramified central division algebra $D_0$  over $F$
   and  $E/F$ an unramified  cyclic extension  with $\sigma$ a generator of Gal$(E/F)$ (\cite[Lemma 5.14]{JW}, cf. 
\cite[Lemma 4.1]{PPS}). 
By (\cite[Theorem 5.15]{JW}), the center  $Z(\bar{D})$  of $\bar{D}$ is isomorphic to the residue field of $E$.
Since ind$(D) =$ ind$(D_0 \otimes E)[E : F] $ (\cite[Lemma 4.2]{PPS}), it follows that   $E$ is a subfield of $D$
 (\cite[Theorem 2.9]{JW}).  

Let $e (D/F) =  \nu_D(\pi)$ be the ramification index of $D/F$. Then $e  = [ E : F]$ (\cite[Theorem 13.3]{reiner} and
 \cite[Proof of Theorem 5.15]{JW}). 

We record here the following well known fact. 
\begin{lemma}
\label{ram-index}
 Suppose  $E_1 \subset F$ is a subfield with $F\subset E_1$.
 Let $D_1$ be the commutant of $E_1$ in $D$.  Then $ e(D/F) = e(D_1/E_1)[ E_1 : F]$.
 Further if $\pi_D \in D$ is a parameter with $\pi_D^{[E_1 : F]} \in D_1$, then $\pi_D^{[E_1 : F]} \in D_1$
 is a parameter in $D_1$.
 \end{lemma}
 
 \begin{proof}
 We write  $D = D_0 + (E, \sigma, \pi) \in Br(F)$ for some  unramified central division algebra $D_0$  over $F$
   and  $E/F$ an unramified  cyclic extension  with $\sigma$ a generator of Gal$(E/F)$
 Since $ D_1 =   D \otimes_F E_1 = D_0 \otimes_F E_1 + (E, \sigma, \pi) \otimes_F E_1$, 
 by (\ref{cyclic}), we have $D_1 = D_0 \otimes_F E_1 + (E/E_1, \sigma^{[E_1 : F]}, \pi)$.
 Since $E/F$ is unramified, $E_1/F$ is unramified and hence $\pi \in E_1$ is a parameter. 
 Since $D_0$ is unramified over $F$, $D_0 \otimes_F E_1$ is unramified over $E_1$. 
 In particular     $e(D_1/F_1) =  [E : E_1] =  [ E: F]/[E_1 : F] = e(D/F)/[E_1 : F]$. 
 
 Let $\Lambda$ be the unique maximal order in $D$ and $\Lambda_1$ be the unique maximal order in $D_1$. 
 Then $\Lambda_1 = \Lambda \cap D_1$. 
 Let $\pi_D$ be a parameter in $D$. Suppose that $\pi_D^{[E_1 : F]} \in D_1$. 
 Since $ \pi = u \pi_D^{e(D/F)}$ for some unit in $\Lambda$, 
 we have  $ \pi = u \pi_D^{e(D_1/F_1)[E_1 : F]} = u ( ( \pi_D)^{[E_1 : F]})^{e(D_1/F_1)}.$
 Since $\pi_D^{[E_1 : F]}, \pi \in D_1$, $u \in D_1 \cap \Lambda = \Lambda_1$. 
 Hence $\pi_D^{[E_1 : F]}$ is a parameter. 
 \end{proof}

Let $F_0$  be a complete discretely valued field with residue field  $K_0$  of characteristic not 2
 and $F/F_0$ a field extension of degree  2.      Let $\pi_0 \in R_0$  and  $\pi \in R$ be parameters. 
   If $F/F_0$ is unramified, then we chose $\pi = \pi_0$. 
If $F/F_0$ is ramified, then we choose $\pi_0$  such that $F = F_0(\sqrt{\pi_0})$ and let $\pi = \sqrt{\pi_0}$.

Let $D$ be a central division algebra over $F$ with a
 $F/F_0$-involution. 
  Suppose  $\pi_D \in \Lambda$  is a  parameter with $\tau(\pi_D) =  \pm  \pi_D$.  
Let $h$ be a hermitian form over $(D,\tau)$. Since every element in $D^*$ is of the form 
$\pi_D^ru$ for some $u \in \Lambda$ a unit and $r \in \Z$ and $\pi_D u = u' \pi_D$ for some unit $u' \in \Lambda$, 
$h = h_1 \perp  \pi_D h_2$ with $h_1 = < u_1, \cdots , u_r> $ and $h_2 =  < u_{r+1}, \cdots , u_n>$
for some units $u_i \in \Lambda$. Note that $h_1$ is an 
hermitian form over $(D, \tau)$ and $h_2$ is a hermitian or a skew hermitian form over $(D, int( \pi_D) \tau)$.
 Since $\Lambda$ is the unique maximal $R$-order,   $\tau(\Lambda) =   \Lambda$ and hence 
 $\tau$ induces an involution $ \bar{\tau}$  on $\bar{D}$. Similarly 
 $int(\pi_D) \tau$ also induces   an  involution on $\bar{D}$. 
 For any unit $u \in \Lambda$, let $\bar{u}$ denote its image in $\bar{\Lambda}$.
 Let $\bar{h}_i$ be the image of $h_i$. 
 By (\cite{Larmour}),   $h$ is isotropic if and only if $h_1$ or $h_2$ is isotropic if and only if 
 $\bar{h}_1$ or  $\bar{h}_2$ isotropic. 
 In particular $h$ is hyperbolic  if and only if $h_1$ and $h_2$ are hyperbolic if and only if 
 $\bar{h}_1$ and $\bar{h}_2$ are   hyperbolic.

\begin{prop} 
\label{ramified} Suppose $F/F_0$ is ramified and $K_0$ is either a local  field or
a global field with no real places. 
 Let $h$   be a  hermitian form  over $(D, \tau)$. If dim$(h)$ is even, disc$(h)$ is trivial and 
 Rost invariant $R(h)$  of $h$ is trivial, then $h$ is hyperbolic. 
 \end{prop} 
 
 \begin{proof}   
  Since $K_0$ is a local field or a global field with no real places, cd$_2(F_0) \leq 3$. 
 If $D = F$, then  by (\ref{odd}),  $h$ is hyperbolic. 

 Suppose that  $D \neq F$. 
 Since $F/F_0$ is ramified, we have $F = F_0(\sqrt{\pi_0})$ for some parameter $\pi_0 \in F_0$.
 By (\cite[Lemma 6.3]{PS2022}),   $D = D_0 \otimes _{F_0} F$ for some unramified central simple  algebra $D_0$ over $F_0$
 with per$(D_0)  = 2$. Since $K_0$ is a global field or a local field, ind$(D_0) = $ per$(D_0) = 2$.
   
 Suppose that dim$(h)$ is even, disc$(h)$ is trivial and 
 Rost invariant $R(h)$  of $h$ is trivial.  Since $h$ is hyperbolic if and only if 
 $uh$ is hyperbolic for any $u \in D^*$ with $\tau(u) = \pm u$,  by (\ref{two-inv}), 
 we assume that $\tau = \star  \otimes \iota$, where $\star$ is the canonical involution on the quaternion algebra $D_0$
 and $\iota$ the non trivial automorphism of $F/F_0$. 
 
 Since $F/F_0$ is ramified, we have $F = F_0(\sqrt{\pi_0})$ for some parameter $\pi_0 \in F_0$. 
 Let $\Lambda_0$ be the maximal order of $D_0$.  Then $\Lambda = \Lambda_0 \otimes_{R_0}R$ is a
 maximal order of $D$.  Since $D$ is unramified, $ \pi= \sqrt{\pi_0}$ is a parameter in $\Lambda$. 
 We have $h = h_1 \perp \pi h_2$, with $h_1 = <u_1, \cdots, u_r>$ and $h_1 = <u_{r+1} , \cdots , u_n>$
 for some $u_i \in \Lambda$ with $\tau(u_i) = u_i$ for  $ 1 \leq i \leq r$ and $\tau(u_i) =-  u_i$ for  $ r + 1\leq i \leq n$. 
 Since $\bar{D} = \bar{D}_0$,  we have  units $v_i \in \Lambda_0$ with $\bar{v}_i = \bar{u}_i$ for all $i$. 
Since $\tau(u_i) = u_i$ (resp. $\tau(u_i) = -u_i$) for $1 \leq i \leq r$ (resp. $r+1 \leq i \leq n$), 
replacing $v_i$ by $( v_i + \tau(v_i))/2$ (resp. $v_i - \tau(v_i)$), we assume that $\tau(v_i) = v_i$ (resp $\tau(v_i) = -v_i$).
Since $F_0$ is complete, we have $<v_i> \simeq <u_i>$ as hermitian (resp. skew hermitian)  forms for  $1 \leq i \leq r$
(resp. $r + 1\leq i \leq n$). Hence
replacing $u_i$ by $v_i$, we assume that $h_i$ are  defined over $D_0$. 

 By (\ref{rost-iso}), without loss of generality we assume that  $h$ is anisotropic. 
 In particular   $h_1$ and $h_2$ are anisotropic over $(D_0, \star)$.

Since $K_0$ is a local field or a global field with no real places, by (\cite[Ch.10, Theorem 3.6 and Theorem 4.1]{Sch},), 
 the dimensions  of anisotropic skew-hermitian forms  over 
$(\bar{D}_0, \star)$ are at most 3. 
Since $F_0$ is complete,    dim$(h_2) \leq 3$.
By   (\ref{dim3}), $<1, -\pi_0>h_2$ is hyperbolic as skew-hermitian form over $(D_0, \star)$.
Since $h_2$ is given by units in $\Lambda_0$ and $\pi_0$ is a parameter in $\Lambda_0$, 
$h_2$ is hyperbolic (\cite{Larmour}).  Hence $h = h_1$. Since  any quadratic form over 
$K_0$ of dimension at least 5 is isotropic and dim$(h_1) = dim(h)$ is even, $h_1$ is hyperbolic as a hermitian form
over $(D_0, \star)$ (cf. \cite[Ch. 10, Examples 1.8(ii)]{Sch}). Hence $h$ is hyperbolic.  
  \end{proof}

Suppose that $F/F_0$ is unramified. Then $D = D_0 + (E, \sigma, \pi_0)$  for some unramified 
central division algebra $D_0$ over $F$ and $E/F$ a cyclic extension.  
 Further $D_0$ and $(E, \sigma, \pi_0)$ have  $F/F_0$-involutions  (\cite[Lemma 6.4]{PS2022}).
 Since cores$_{F/F_0}(E, \sigma, \pi_0) = $ cores$_{F/F_0}(E, \sigma) \cdot  (\pi_0 )$ (\cite[Proposition 1.5.3]{Neukirch}), 
 by taking the residue, we have  cores$_{F/F_0}(E,\sigma) = 0$.
 Hence $E/F_0$ is a dihedral extension (cf. \cite[Proposition 3.2]{PS2022}). 
 In particular we have an automorphism of $E/F_0$ of order 2 which is not identity on $F$.
 Since $D$ has a $F/F_0$-involution, there exists an involution $\tau$ on $D$ with 
 $\tau(E) = E$ (\cite[Ch.8, Theorem 10.1]{Sch}).  For the rest of the section we assume that $\tau$ is a $F/F_0$-involution
 on $D$ with $\tau(E) = E$.

\begin{lemma} 
  \label{pid} Suppose $F/F_0$ is unramified. 
  Then there exists a parameter $\pi_D \in \Lambda$ such that 
 $\tau(\pi_D) = \pm \pi_D$, 
 int$(\pi_D)$  restricted to $E$ is  a generator of Gal$(E/F)$.  
 \end{lemma} 
 
 \begin{proof} Let $\sigma$ be a generator of Gal$(E/F)$. Then, by (\cite[Lemma 8.1]{Suresh2024}), 
  there exists a   parameter $z$ in $\Lambda$  such that 
$int(z)$ restricted to $E$ is a generator of Gal$(E/F)$.  

Since 
$z Ez^{-1}  = E$ and  $\tau(E) = E$, we have $\tau(z) E \tau(z)^{-1}  = E $.  
Since $Nrd(\tau(z)) = \tau(Nrd(z))$,  it follows that 
$\tau(z)$ is also a parameter. Hence $\tau(z) = u_2 z$ for some unit $u_2 \in \Lambda$. 
Since int$(z)(E) = E$ and $\tau(z) E \tau(z)^{-1}  = E $, we have $u_2E u_2^{-1} = E$. 
Let $\theta$ be the automorphism of $E$ given by the restriction of int$(u_2)$ to $E$. 
Since $u_2$ is a unit in $\Lambda$, $\bar{u}_2 \in \bar{D}$ is nonzero.
Since the residue field  $L$ of $E$ is the center $Z(\bar{D})$ of $\bar{D}$,  $\theta$ induces the identity map 
on the residue field of $E$. Since $E/F$ is unramified, $\theta$ is  identity (\cite[p. 26, Corollary]{CFANT}). Hence $u_2 \in D_1$. 

Suppose that  the image of $u_2 $ is not equal to $-1$ in $\bar{D}$. 
The $1 + u_2$ is a unit in $\Lambda$. 
Since $1 + u_2 \in D_1$, the restrictions of 
int$(1  +  u_2)z$ and int$(z)$ to $E$ coinside.  Let $\pi_D = z + \tau(z) = (1 + u_2)z$.
Then $\tau(\pi_D) = \pi_D$ and int$(\pi_D)$ restricted to $E$ is a generator of Gal$(E/F)$.  

Suppose that the image of $u_2$ is equal to $-1$ in $\bar{D}$.
Let $\pi_D = z - \tau(z)$. Then $\tau(\pi_D) = - \pi_D$ and, as above, 
int$(\pi_D)$ restricted to $E$ is a generator of Gal$(E/F)$. 
\end{proof}

\begin{prop} 
\label{unramified} Suppose that 
  $K_0$ is  a local field or a global field with no real places.   Suppose that 
  $F/F_0$ and $D$ are unramified.
  Let $h$   be a  hermitian form  over $(D, \tau)$. If dim$(h)$ is even, disc$(h)$ is trivial and 
 Rost invariant $R(h)$  of $h$ is trivial, then $h$ is hyperbolic. 
\end{prop}

\begin{proof}  Since  
  $K_0$ is  a local field or a global field with no real places, cd$_2(K_0) \leq 2$ (\cite[p. 83, p. 86, p. 87]{SerreGC}).
  Since $F_0$ is complete with residue field $K_0$,   cd$_2(F_0) \leq 3$ (\cite[Ch II, Proposition 12]{SerreGC}). 
   By  (\ref{odd}), we assume that deg$(D)$  is even. 

 Let  $\Lambda$ be the maximal order of $D$ and $\pi_0 \in F_0$ be a parameter. 
 Since  $D$  and $F/F_0$ are  unramified,  $ \pi_0 \in \Lambda$ is a parameter. 
 Write $h = h_1 \perp \pi_0 h_2$ with $h_1$ and $h_2$ given by units in $\Lambda$. 
Since  $  \pi _0  \in F_0$, $h_1$ and $h_2 $ are hermitian forms over $D$.
Since dim$(h)$ is even,  dim$(h_1) = $ dim$(h_2)$ modulo 2. 
Since  disc$(h)$ is trivial  and $\pi_0 \in F_0$,  disc$(h_1) = $ disc$(h_2)$. Since $K_0$ is a global field with no real places or a local field,
 $\bar{h}_1 \simeq \bar{h}_2$ (\cite[p. 375/376]{Sch}). Hence $h_1 \simeq  h_2$  and $h  = <1, \pi_0> h_1$. 
By (\ref{ri-lambda}), we have  $R( <1, \pi_0> h_1) =   \DD(h_1) \cdot  (-\pi_0)$, where $\DD(h_1)$ is the discriminant algebra
of $h_1$.   Since $R(h)$  is trivial, by taking the residue, we see that 
$\DD(\bar{h}_1)$ is trivial. Hence, by (\ref{dalgebra}), $\bar{h}_1$ is hyperbolic. In particular 
 $h_1$ is hyperbolic and hence   $h$ is 
hyperbolic. 
\end{proof}

\begin{theorem}
\label{class-complete} Let $F_0$  a complete discretely valued field with residue field $K_0$  a global field with 
no real places or a local field 
 and $F/F_0$ a quadratic extension.  Let $D$ be a central division  algebra over $F$ with a $F/F_0$-involution 
 $\tau$.   Suppose that per$(D)$ is coprime to char$(K_0)$.
 Let $h$   be a  hermitian form  over $(D, \tau)$. If dim$(h)$ is even, disc$(h)$ is trivial and 
 Rost invariant $R(h)$  of $h$ is trivial, then $h$ is hyperbolic. 
\end{theorem}

\begin{proof}  If $F/F_0$ is ramified, then $h$ is hyperbolic (\ref{ramified}). 
  If $F/F_0$ is unramified and $D$ is unramified, then $h$ is hyperbolic (\ref{unramified}).

Suppose that $F/F_0$ is unramified and $D$ is ramified.   Hence $\pi = \pi_0\in F_0$. 
We have 
$D = D_0 + (E, \sigma, \pi) \in Br(F)$ for some  unramified central division algebra  over $F$   and 
$E/F$ an unramified  cyclic extension  with $\sigma$ a generator of Gal$(E/F)$. 
Then cores$_{F/F_0}(D_0)$ and cores$_{F/F_0}(E, \sigma, \pi)$ are trivial  (\cite[Lemma 6.4]{PS2022}) and hence 
$D_0$ and $(E, \sigma, \pi)$ 
 have  $F/F_0$-involutions.  Since  $F_0$ is complete and 
 cores$_{F/F_0}(E, \sigma, \pi) = $ cores$_{F/F_0}(E, \sigma) \cdot(\pi)$,
 by taking the residues we see that cores$_{F/F_0}(E, \sigma) = 0$.

Suppose $[E : F]$ is odd.   Since $E/F_0$ is a dihedral extension (\cite[Proposition 3.2]{PS2022}), 
there exists a subfield $L$ of $E$ such 
that  $F_0 \subset L \subset E$, $[ L  : F_0] = [E : K]$ and $LF  = E$ (cf. \cite[Lemma 3.2]{PS2022}). 
Since $D\otimes _{F_0} L \simeq D\otimes_F E = D_0 \otimes_F E$ is 
unramified,  by the unramified case, $h$ is hyperbolic over $D\otimes L$. Since 
$[L : F_0] = [E : F]$ is odd,  $h$ is hyperbolic (\cite{BL}). 

Suppose $[E : F]$ is  even.  Let $E_1/F$ be the unique subextension of $E/F$ with $[ E_1 : F] = 2$. 
Let $D_1$ be the commutant of $E_1$ in $D$. 
Let $S_1$ be the integral closure of $R$ in $E_1$ and $\Lambda_1$ be the unique maximal $S_1$-order of
$D_1$.  Let $\pi_D \in D$ be as in (\ref{pid}).  Since int$(\pi_D)$ restricted to $E$ is the  automorphism $\sigma$ of 
$E/F$, int$(\pi_D)^2$ restricted to $E_1$ is identity.  Hence   $\pi_D^2 \in D_1$.  

Let $\nu_D$ be the discrete valuation on $D$. Since $\pi_D$ is a parameter  and
$e = [ E : F]$ is the ramification index of $D$, we have $\pi = u \pi_D^e$ for some unit $u \in \Lambda$.
Since $E_1 \subset E$,  $\pi_D^2$ is a parameter in $D_1$ (cf. \ref{ram-index}). 
Write $h = h_1 \perp  \pi_D h_2$ with $h_1 = < u_1, \cdots , u_r> $ and $h_2 =  < u_{r+1}, \cdots , u_n>$
for some units $u_i \in \Lambda$.
Since   $D\otimes E_1$ is Brauer equivalent to the division algebra $D_1$ (\cite[Ch.8, Theorem 5.4]{Sch}), 
by (\cite[Lemma 1.8]{JW}) $\bar{\Lambda}_1 = \bar{\Lambda} $. 
As in the proof of (\ref{ramified}),  we assume that      $u_i \in \Lambda_1 \subset D_1$ for all $i$. 

Since $E_1/F_0$ is a dihedral extension of degree 4, there exists a subfield $L$ of $E_1$ such that 
$F_0 \subset L$ and $LF = E_1$. Let $\pi_{D_1} = \pi_D^2 \in D_1$.
 Let $a \in D_1$ with $\tau(a) = a$ and consider the hermitian form $<a>$.
 Let $\phi = int(\pi_D)$. 
Then, by (\ref{morita-cor}), $<a \otimes 1>$ corresponds to the rank two hermitian form 
$$
\begin{pmatrix}
a & 0 \\ 0 &  \pi_{D_1}^{-1} \phi^{-1}(a)
\end{pmatrix}
$$
over $D_1$ under the isomorphism
 $D\otimes_{F_0} L \simeq D\otimes_FE_1 \simeq M_2(D_1)$ and Morita equivalence. 
 Similarly if $<\pi_D a>$ is a   hermitian form over $D$ with $a \in D_1$, then it corresponds to 
 the hermitian form
  $$
  \begin{pmatrix}
1 & 0 \\ 0 &     \pi_{D_1}^{-1}
\end{pmatrix}
  \begin{pmatrix}
0 &   \pi_{D_1}  \\  1  & 0   
\end{pmatrix}
\begin{pmatrix}
a & 0 \\ 0 &   \phi^{-1}(a)
\end{pmatrix} = \begin{pmatrix}
0 &    \pi_{D_1}\phi^{-1}(a)  \\      \pi_{D_1}^{-1}  a & 0 
\end{pmatrix}
$$
over $D_1$.  In particular, it is hyperbolic. 
Thus the the form $h = h_1 \perp   \pi_{D} h_2$ corresponds, under the Morita equivalence, 
 to $<u_1, \cdots, u_r> \perp   \pi_{D_1}^{-1} 
<\phi^{-1}(u_1), \cdots, \phi^{-1}( u_r)>
\perp h_3$ over $D_1$, with $h_3$ hyperbolic.   

Since the dimension, discriminant and the Rost invariant of $h$  are trivial,  it follows that 
the dimension, discriminant and  the Rost invariant of $<u_1, \cdots, u_r> \perp   \pi_{D_1}^{-1} 
<\phi^{-1}(u_1), \cdots, \phi^{-1}( u_r)>$  are trivial (\ref{rost-iso}). 
 Since the ramification index of $D_1$ is strictly less than the ramification index 
of $D$ (cf. \ref{ram-index}), by induction,  $<u_1, \cdots, u_r> \perp   \pi_{D_1}^{-1} 
<\phi^{-1}(u_1), \cdots, \phi^{-1}( u_r)>$  is hyperbolic over $D_1$. 

Let $u \in \Lambda_1$ be a unit with $\tau(\pi_{D_1}^{-1} u) = \pi_{D_1}^{-1}  u$. 
Since $\tau(\pi_{D_1}) = \pi_{D_1}$,    we have 
$$
\begin{array}{rcl}
 <\pi_{D_1}^{-1} u>  & \simeq  &  < \pi_{D_1} ( \pi_{D_1}^{-1} u)  \tau(\pi_{D_1}) >  \\
 & \simeq &  < \pi_{D_1} ( \pi_{D_1} ^{-1}u)   \pi_{D_1} >  \\
 & \simeq &  < \pi_{D_1} ( \pi_{D_1}^{-1} u \pi_{D_1}) >  \\  
 \end{array}
  $$
Let $u' = \pi_{D_1}^{-1} u \pi_{D_1}$. Since $u \in \Lambda_1$ is a unit, $u' \in \Lambda_1$ is a unit
and $\pi_{D_1}^{-1}<u> \simeq \pi_{D_1}<u'>$.
Hence  $\pi_{D_1}^{-1}  <\phi^{-1}(u_1), \cdots, \phi^{-1}( u_r)> \simeq \pi_{D_1}<u_1', \cdots , u_r'>$ for some
units $u_i' \in \Lambda_1$.  Since $<u_1, \cdots, u_r> \perp  \pi_{D_1}<u_1', \cdots , u_r'>$ is  hyperbolic  over $D_1$ 
and  $\pi_{D_1}$ is a parameter in $\Lambda_1$,   $<u_1, \cdots, u_r>$ is hyperbolic over $D_1$
(\cite{Larmour}). 
 In particular   $h_1 = <u_1, \cdots, u_r>$ is  hyperbolic over $D$.  Hence we have $h= \pi_D h_2$.
Let $\tau' = $ int$(\pi_D)$. Then $h_2$ is hermitian form over $(D, \tau')$. Since the diagonal entries of $h_2$ are units, 
as above, $h_2$ is hyperbolic. Hence $h$ is hyperbolic. 
\end{proof}

\begin{theorem}
\label{class-semiglobal} Let $K$ be a local field with residue field $\kappa$ and
 $F$  the function field of a  curve over $K$. Let $A$ be a central simple algebra over $F$ with 
 an involution  $\tau$ of second kind. Suppose that 2per$(D)$ is coprime to char$(\kappa)$.
  Let $h$  be a hermitian form over  $(A, \tau)$. 
 If dim$(h)$ is even,   disc$(h)$ is trivial   and $R(h)$ is trivial, then 
 $h$ is hyperbolic.
 \end{theorem}
 
 \begin{proof} Suppose dim$(h)$ is even,   disc$(h)$ is trivial   and $R(h)$ is trivial.
 Let $F_0 = F^\tau$ and $\Omega$ be the set of  divisorial discrete valuations of $F_0$. 
 By the local global principal for hermitian forms (\cite[Theorem 11.6]{PS2022}), it is enough to show that 
 $h$ is hyperbolic over the completions $F_{0\nu}$ for all $\nu \in \Omega$. 
 
 Let $\nu \in \Omega$. Suppose $F\otimes_{F_0} F_{0\nu}$ is not a field.
 Then, by (\cite[Proposition 2.14]{KMRT}),  $h$ is hyperbolic over $F_{0\nu}$. 
 Suppose $F\otimes_{F_0} F_{0\nu}$ is   a field. Then, by (\ref{class-complete}), 
 $h$ is hyperbolic. 
 \end{proof}

 \section{Rost Injectivity}
 \label{main}

 \begin{prop}
\label{cores-ram} Let $F_0$  a complete discretely valued field with residue field  $K_0$
 and  $F/F_0$ a  ramified quadratic extension.  Let $A$ be a central simple algebra over $F$ with 
 a $F/F_0$-involution.   Suppose that $2 per(A)$ is coprime to char$(K_0)$.  
 Let $\mu \in F^*$.  If cores$_{F/F_0}(A \cdot (\mu)) = 0$,  then  there exist $\lambda \in F_0$ such that 
 $ A \cdot (\mu) = A \cdot (\lambda)$.
 \end{prop}
 
 \begin{proof}  Since $F/F_0$ is ramified, by (\cite[Lemma 6.3]{PS2022}), there exists an unramified central simple algebra 
 $A_0$ over $F_0$  with  per$(A) \leq 2$ such that $A  \simeq A_0 \otimes_{F_0}F$. 
   Since $F/F_0$ is ramified and char$(K_0) \neq 2$, there exists a
 parameter   $\pi_0 \in F_0$ such that $F = F_0(\sqrt{\pi_0})$. Let $\pi = \sqrt{\pi_0}$.
Write $\mu = \mu_0 \pi^r$ for some unit $\mu_0$ in the valuation ring of $F$.
Since 2 is coprime to char$(K_0)$, 
 we have $\mu_0 = \lambda \mu_1^2$ for some  $\mu_1 \in F$ and 
$\lambda \in F_0$ a unit in the valuation ring $R_0$.  Since per$(A) \leq 2$, 
$A \cdot (\mu_1^2) = 0$ and hence  
$A \cdot (\mu) = A \cdot (\lambda) + A \cdot (\pi^r)$.  

Since $A = A_0 \otimes_{F_0} F$, we have
$$\begin{array}{rcl} 
cores_{F/F_0}(A \cdot (\mu)  & = &  cores_{F/F_0}(A \cdot (\lambda) + cores_{F/F_0}(A \cdot (\pi^r) \\
& = &  A_0 \cdot  ( N_{F/F_0}(\lambda)) + A_0 \cdot  (N_{F/F_0}(\pi)^r).
\end{array} $$ 
Since $\lambda \in F_0$ and per$(A_0) = 2$, we have  $ A_0 \cdot  (N_{F/F_0}(\lambda)) = 
A_0 \cdot (\lambda^2)= 0$.
Since $\pi = \sqrt{\pi_0}$, we have $N_{F/F_0}(\pi) = -\pi_0$ and 
hence $A_0 \cdot (N_{F/F_0}(\pi)^r) = A_0 \cdot ((-\pi_0)^r)$. 
Since cores$_{F/F_0}(A \cdot (\mu) = 0$,  $A_0 \cdot ((-\pi_0)^r) = 0 $.

Since $A_0$ is unramified, there is an Azumaya algebra $\AA_0$ over $R_0$
such that $\AA_0 \otimes_{R_0}F_0 \simeq A_0$. Let $\bar{A}_0 = \AA_0 \otimes_{R_0} K_0$.    
Since   $A_0 \cdot ((-\pi_0)^r)   = 0$,
by taking the residues, we see that  $\bar{A}_0^r$ is trivial.  Since $F_0$ is complete, $A_0^r$ is trivial
and hence $A^r = 0$. 

Thus  $A \cdot (\mu) =  A \cdot (\lambda) + A \cdot (\pi)^r = A \cdot (\lambda) + A^r \cdot (\pi) =  A \cdot (\lambda) $. 
\end{proof}

\begin{prop}
\label{cores} Let $F_0$  a complete discretely valued field with residue field  $K_0$ a global field with no real places or 
a local field
 and $F/F_0$ a quadratic extension.  Let $A$ be a central simple algebra over $F$ with a $F/F_0$-involution 
 $\tau$.   Suppose that $2 per(A)$ is coprime to char$(K_0)$.  
 Let $\mu \in F^*$. If   cores$_{F/F_0}( A \cdot (\mu)) = 0$, then there exists $\lambda \in F_0$ such that 
$ A \cdot (\mu) = A \cdot (\lambda)$.
 \end{prop}
 
 \begin{proof}  If  $F/F_0$ is ramified,  this follows from  (\ref{cores-ram}). 
 
 Suppose that $F/F_0$ is unramified.   Let $\pi_0$ be a parameter in $F_0$.
 Then $\pi_0$ is also a parameter in $F$.
 We have   $A = A_0 + (E, \sigma, \pi_0)$ for some unramified  algebra  $A_0$ over $F$ 
 and  a cyclic extension $E/F$ which is unramified (\cite[Lemma 5.14]{JW}, cf. 
\cite[Lemma 4.1]{PPS}). Further cores$_{F/F_0}(A_0) = 0$ and 
 $E/F_0$ is a dihedral extension  (cf. \cite[Proposition 3.2 and Lemma 6.4]{PS2022}).

 Let  $K$ and $L$  be the residue fields of  $F$ and $E$ respectively. 
 Then $K/K_0$ is a quadratic extension, $L/K$ a cyclic extension and 
 $L/ K_0$ a dihedral extension.
 Let $\sigma_0$ be the automorphism of $L$ induced by $\sigma$. 
 
  Let $n = per(A)$. 
 Since   cd$_2(K) \leq 2$, $n$ is coprime to char$(K)$ and $F$ is complete,  $H^3_{nr}(F, \mu_n^{\otimes 2}) 
 \simeq H^i(K, \mu_n^{\otimes 2}) = 0$ (cf. \cite[7.10]{serrecohinv}).
 In particular, since $A_0$ is unramified, 
 for any unit $u $ in the valuation ring of $F$, $A_0 \cdot ( u ) = 0$. 
 
 We  write $\mu = \mu_0  \pi_0^r$ for some $\mu_0 \in F$ a unit in the valuation ring. 
Since $A_0 \cdot (\mu_0) = 0$, we have 
 $$ A \cdot (\mu) = A _0 \cdot (\pi_0^r)  + (E,\sigma, \pi_0) \cdot (\mu_0 \pi_0^r).$$
 Since $\pi_0 \in F_0$, we have cores$_{F/F_0}( A_0\cdot (\pi_0^r)) = cores_{F/F_0}(A_0) \cdot (\pi_0^r)$.
 Since cores$_{F/F_0}(A_0) = 0$, we have  cores$_{F/F_0}( A_0\cdot (\pi_0^r)) = 0$.
 Hence 
 $$ cores_{F/F_0}(A \cdot (\mu)) = cores_{F/F_0}( (E, \sigma, \pi_0) \cdot (\mu_0 \pi_0^r)).$$
 Since 
 $$(E,\sigma, \pi_0) \cdot (\mu_0\pi_0^r) = (E, \sigma, \pi_0) \cdot (\mu_0) = (E, \sigma, \mu_0) \cdot (\pi_0^{-1})$$
 and cores$_{F/F_0}(A \cdot (\mu)) = 0$,  by taking the residues,  we get   
  cores$_{K/K_0}(L, \sigma_0, \bar{\mu}_0) = 0$, where $\bar{\mu}_0$ denotes the image of $\mu_0$ in $K_0$. 
  
 Since $K_0$ is a global field, by (\ref{globalfield}), there exists $\lambda_0 \in K_0$ such that 
 $(L, \sigma_0, \bar{\mu}_0) = (L, \sigma_0, \lambda_0)$.  Let $\lambda_1 \in F_0$ be a lift of 
 $\lambda_0$.  Since $F$ is complete,  $(E, \sigma, \mu_0) = (E, \sigma, \lambda_1)$. 
 In particular  $(E, \sigma, \pi) \cdot ( \mu_0) = (E, \sigma, \pi) \cdot ( \lambda_1).$
   
Since 
 $\mu_0, \lambda_1$ are units in the valuation ring,  we have $A _0 \cdot (\mu_0  ) = A_0 \cdot (\lambda_1 ) = 0$. 
 
 Let $\lambda = \lambda_1 \pi^r \in F_0$. 
Since $(E, \sigma, \pi) \cdot (\mu_0 \pi^r) = (E,\sigma, \pi) \cdot ( \lambda_1 \pi^r) =  (E,\sigma, \pi) \cdot ( \lambda)$. 
We have 
$$
\begin{array}{rcl}
   A \cdot (\mu) &  = & A _0 \cdot  (\mu) +  (E, \sigma, \pi) \cdot ( \mu ) \\
&  =  & A _0 \cdot  (\mu_0  \pi^r) +  (E, \sigma, \pi) \cdot ( \mu_0 \pi^r) \\
&  =  & A _0 \cdot  (\mu_0) + A _0 \cdot  (\pi^r)  +  (E, \sigma, \pi) \cdot ( \mu_0  )  +  (E, \sigma, \pi) \cdot (   \pi^r) \\
&  =  & A _0 \cdot  (\lambda_1) + A _0 \cdot  (\pi^r)  +  (E, \sigma, \pi) \cdot ( \lambda_1  )  +  (E, \sigma, \pi) \cdot (   \pi^r) \\
& = &   A_0 \cdot (\lambda_1 \pi^r ) + (E,\sigma,\pi) \cdot (\lambda_1 \pi^r)  \\
& = &   A_0 \cdot (\lambda ) + (E,\sigma,\pi) \cdot (\lambda )  \\
& = &  A \cdot (\lambda).
\end{array}$$ 
 \end{proof}

 Let $F_0$ be a field of characteristic not 2 and  $F/F_0$  a quadratic  extension. Let  $A$ be 
 a central simple algebra over $F$ with a $L/K$-involution $\tau$.  Let $h$ be a hermitian form over $(A,\tau)$. 
  The   short  exact sequence of algebraic groups 
 $$ 1 \to SU(A, \tau, h) \to U(A, \tau, h) \to R^1_{F/F_0}\G_m \to 1 $$
 gives the   exact sequence of cohomology sets
 $$  F^{*1} \buildrel{\delta}\over{\to} H^1(F_0, SU(A, \tau, h)) \to H^1(F_0, U(A, \tau, h)).  $$
 
\begin{prop}
\label{center} Let $F_0$  a complete discretely valued field with residue field a global field 
with no real places or 
a local field
 and $F/F_0$ a quadratic extension.  Let $A$ be a central simple algebra over $F$ with a $F/F_0$-involution 
 $\tau$ and $G = SU(A, \tau, h)$.  Let $\theta \in F^{*1}$. 
 If $R_G(\delta(\theta)) = 0$, then $\delta(\theta) $ is trivial.  \end{prop}
 
 \begin{proof} Let  $\mu \in  F ^*$ with $\theta =  \mu \bar{\mu} ^{-1}$. 
Suppose $R_G(\delta(\theta)) = 0$.  Since
$R_G(\delta(\theta)) = $ cores$_{F/F_0}( A \cdot (\mu))$ (\cite[Appendix]{PP}),
 by (\ref{cores}), there exists $\lambda \in F_0$ such that 
$A \cdot (\mu) = A \cdot (\lambda)$.  In particular $\lambda^{-1} \mu \in $ Nrd$(A)$.
Since $\theta = \lambda^{-1} \mu ~ \overline{(\lambda^{-1} \mu})^{-1}$, replacing $\mu$ by $\lambda^{-1}\mu$ we assume that 
$\mu$ is a reduced norm from $A$. Hence,  $\theta$ is in the image of $U(A,\tau, h)(F_0)$ (cf. \cite[p. 202]{KMRT}).
In particular   $\delta(\theta)$ is trivial in $H^1(F_0, SU(A, \tau))$.  
 \end{proof}

 \begin{theorem} 
 \label{complete} Let $F_0$  a complete discretely valued field with residue field a global field 
 with no real places or 
a local field 
 and $F/F_0$ a quadratic extension.  Let $A$ be a central simple algebra over $F$ with a $F/F_0$-involution 
 $\tau$ and $h$ a hermitian form over $(A, \tau)$. 
  Then the Rost invariant $R_{SU(A,\tau, h)} : H^1(F_0, SU(A, \tau, h)) \to H^3(F_0, \Q/\Z(2))$  has trivial kernel. 
 \end{theorem}
 
 \begin{proof} Let $\zeta \in  H^1(F_0, SU(A, \tau, h))$ with $R_{SU(A, \tau, h)}(\zeta) = 0$. 
 Let $h'$ be a hermitian form over $(A, \tau)$ representing the image of $\zeta$ in $H^1(F_0, U(A, \tau, h))$.
 Then $d_h(h') = 1$ and dim$(h) = $ dim$(h')$.  Since $R_{SU(A, \tau, h)}(\zeta) = 0$, $R_h(h')= 0$. 
 Since $R(h\perp -h') = R_h(h')$ (\ref{sumrost}), $R(h\perp -h') = 0$. Hence, by (\ref{class-complete}), $h \perp - h'$ is hyperbolic and 
 hence $h \simeq h'$. In particular, 
 $h'$ represents the trivial element in $H^1(F_0, U(A, \tau, h)$ and hence  by the exact sequence of cohomology sets, 
 there exists 
 $\theta \in F^{*1}$ such that $\delta(\theta) =  \zeta$. 
 Hence, by (\ref{center}),  $\zeta$ is trivial. 
 \end{proof}

   \begin{theorem} 
   \label{main-rost-su} Let $K$ be a local field with residue field $\kappa$. 
   Let $F_0$  the  function field  of a curve over   $K$
 and $F/F_0$ a quadratic field extension.  Let $A$ be a central simple algebra over $F$ with a $F/F_0$-involution 
 $\tau$ and $h$ a hermitian form over $(A, \tau)$.  Suppose that $2per(A)$ is coprime to char$(\kappa)$. 
  Then the Rost invariant $R_G : H^1(F_0, SU(A, \sigma)) \to H^3(F_0, \Q/\Z(2))$  has trivial kernel. 
 \end{theorem}

\begin{proof}  Let $\Omega$ be the set of divisorial discrete valuations of $F_0$. 
For $\nu \in \Omega$, let $F_{0\nu}$ be the completion of $F_0$ at $\nu$. 

By  (\cite[Theorem 13.1]{PS2022}), 
$$H^1(F_0, SU(A, \tau, h) )  \to \prod_{\nu \in \Omega}  H^1(F_{0_\nu}, SU(A, \tau, h) )$$ has trivial kernel. 
By  a result of Kato(\cite[Theorem 5.2]{Ka}),  
 $$H^3(F_0, \Q/\Z(2)) \to \prod_{\nu \in \Omega}H^3(F_{0\nu}, \Q/\Z(2))$$ has trivial kernel. 

Thus, in view of   the  
 following commutative diagram 
$$
\begin{array}{ccc}
H^1(F_0, SU(A, \tau, h) ) & \buildrel{R_{SU(A, \tau, h)}}\over{\to} & H^3(F_0, \Q/\Z) \\
\downarrow & & \downarrow \\
\prod_{\nu \in \Omega}  H^1(F_{0_\nu}, SU(A, \tau, h) ) & \buildrel{R_{SU(A, \tau, h)}}\over{\to} & 
\prod_{\nu \in \Omega}  H^3(F_{0_\nu}, \Q/\Z), 
\end{array}
$$
 it is enough to show that $R_{SU(A, \tau, h)} :  H^1(F_{0_\nu}, SU(A, \tau, h) )   \to 
  H^3(F_{0_\nu}, \Q/\Z)$  has trivial kernel  for all $\nu \in \Omega$. 

Let $\nu \in \Omega$ and     $\kappa(\nu)$  the residue field of $F_{0\nu}$.
Then either $\kappa(\nu)$ is a local field (in fact a finite extension of $K$) or a global field of positive characteristic.

Suppose that $F\otimes_{F_0} F_{0\nu}$ is a field. Then, by (\ref{complete}), 
the Rost invariant map  $R_{SU(A, \tau, h)}  :  H^1(F_{0_\nu}, SU(A, \tau, h) ) \to  
   H^3(F_{0_\nu}, \Q/\Z)$ has trivial kernel.

Suppose that $F\otimes_{F_0} F_{0\nu}$ is not a field. Then 
$ F\otimes_{F_0} F_{0\nu} \simeq F_{0\nu} \times F_{0\nu}$
and $A \otimes_{F_0} F_{0\nu} \simeq (B_\nu, B_\nu^{op}) $  for some centrals simple algebra $B_\nu$
over $F_{0\nu}$ and $SU(A, \tau, h)_{F_{0\nu}} \simeq SL_n(B)$ (\cite[\S 29]{KMRT}). 
Further $R_{SU(A, \tau, h)} : H^1(F_{0\nu}, SU(Am \tau, h)) \to H^3(F_{0 \nu}, \Q/\Z)$ is 
identified with the Rost invariant map $R_{SL_n(B)} : H^1(F_{0\nu}, SL_n(B)) \to H^3(F_{0 \nu}, \Q/\Z)$ (\cite[p. 438]{KMRT}).
Since   $R_{SL_n(B)} : H^1(F_{0\nu}, SL_n(B)) \to H^3(F_{0 \nu}, \Q/\Z)$ has trivial 
kernel (\cite[Theorem 4.12]{PPS}),  $R_{SU(A, \tau, h)} : H^1(F_{0\nu}, SU(A, \tau, h)) \to H^3(F_{0 \nu}, \Q/\Z)$
has trivial kernel.  
 \end{proof}

\begin{theorem}
\label{main-rost}
Let $K$ be a local field with residue field $\kappa$ and
 $F$  the function field of a  curve over $K$. Let $G$ be an absolutely simple simply connected linear algebraic 
  group over $F$ of classical type.
 If char$(\kappa)$ is good with respect to $G$, then 
 $$R_G : H^1(F, G) \to H^3(F, \Q/\Z(2))$$
 is injective. 
\end{theorem} 

\begin{proof}  If $G$ is of type $B_n$, $C_n$ or $D_n$, then by   (\cite{preeti}), $R_G$ is  injective. 
Suppose $G$ is  of type $^1A_n$. Then $G =  SL_1(A)$  for some central simple algebra over $F$ (\cite[Theorem 26.9]{KMRT})
and hence     by  (\cite{PPS}), $R_G$ is injective. 
Suppose $G$ is of type $^2A_n$. Then $G = SU(A, \tau, h)$ for some central simple algebra $A$ over with an
involution $\tau$ and $h$ a hermitian form over $(A, \tau)$ (\cite[Theorem 26.9]{KMRT}). Hence, by (\ref{main-rost-su}), 
$R_G$ has trivial kernel.   Thus, by a twisting argument (\cite[Proposition 28.8]{KMRT}), $R_G$ is injective. 
\end{proof}

\providecommand{\bysame}{\leavevmode\hbox to3em{\hrulefill}\thinspace}


\begin{thebibliography}{10}

\bibitem{Albert}\text{A. Albert}, \textit{Structure of Algebras},  
Amer. Math. Soc. Colloq. Publ., Vol. 24, Amer. Math. Soc., Providence,
RI,  1961, revised printing.

\bibitem{Arason}  \text{J.K. Arason}, 
\textit{Cohomologische invarianten quadratischer {F}ormen},
 {J. Algebra},   {\bf 36} (1975),   {448--491}. 

\bibitem{AEJ}  \text{J.K. Arason, R.Elman and B.Jacob}, \textit{Fields of
    cohomological 2-dimensions three}, Math. Ann. {\bf 274} (1986),
  649-657.  



\bibitem{BL}  \text{E. Bayer-Fluckiger and H.W. Lenstra},  \textit{Forms in odd
degree extensions and self-dual normal bases}, Amer. J. Math. {\bf 112} (1989), 359--373.

\bibitem{BP}\text{E. Bayer-Fluckiger and R. Parimala}, \textit{Galois cohomology of the classical groups over fields of 
cohomological dimension $\leq 2$}, Invent. Math. {\bf 122} (1995), 195-229. 
 
 \bibitem{CFANT}  \text{J.W.S Cassels and A. Fr\"ohlich},
\textit{Algebraic Number Theory}, 
Thomson Book Company Inc, Washington, D.C, 1967.

\bibitem{Chernousov} \text{V.  Chernousov}, 
\textit{ The kernel of the Rost invariant, Serre's Conjecture II and the Hasse principle for quasi-split groups $^{3,6}D_4$, $E_6$,
$E_7$}, 
 Math. Ann. {\bf 326} (2003), 297--330.  
  
  
 \bibitem{CTSB} \text{J.-L. Colliot-Th\'el\`ene},
    \textit{Birational invariants, purity and the {G}ersten conjecture},
 {$K$}-theory and algebraic geometry: connections with
              quadratic forms and division algebras {S}anta {B}arbara,
              {CA}, 1992,
      {Proc. Sympos. Pure Math.},
   {\bf 58}, Part 1 (1995), 
   {1--64}.
         
\bibitem{CTPS1}
\text{J.-L. Colliot-Th\'el\`ene, R. Parimala, V. Suresh},
\textit{Patching and local global principles for homogeneous spaces
over function fields of p-adic curves}, 
Commentari Math. Helv. {\bf 87} (2012), 1011--1033.

\bibitem{Gille-rost-postive} \text{P. Gille},
      \textit{Invariants cohomologiques de {R}ost en caract\'eristique
              positive},
    {$K$-Theory},
   {\bf 21} (2000),   {57--100}. 
         
\bibitem{Skip} \text{Garibaldi, Ryan Skip},
    \textit{The {R}ost invariant has trivial kernel for quasi-split groups
              of low rank}, {Comment. Math. Helv.}, {\bf 76} (2001), {684--711}.  
              
\bibitem{JW} \text{B. Jacob and A. Wadsworth}, \textit{Division
    algebras over Henselian fields},   J. Algebra {\bf 128} (1990),
  126--179. 
 
  \bibitem{Ka} \text{K. Kato}, \textit{A Hasse principle for two-dimensional
global fields, J. reine angew. Math.} {\bf 366} (1986), 142-181.

\bibitem{Knus}  \text{M.-A. Knus}, \textit{Quadratic and {H}ermitian forms over rings},
   Grundlehren der Mathematischen Wissenschaften,    {\bf 294}  (1991) Springer-Verlag, Berlin. 

\bibitem{KMRT}  \text{M.-A. Knus, A.S. Merkurjev, M. Rost and J.-P. Tignol},
 \textit{The Book of Involutions}, A.M.S, Providence RI, 1998. 

 \bibitem{Larmour} D.W. Larmour,
\textit{A {S}pringer theorem for {H}ermitian forms}, Math. Z., {\bf 252} (2006),
459--472. 

\bibitem{Lorenz} \text{F. Lorenz}, \textit{Algebra Volume II: Fields with Structure, Algebras and Advanced Topics},
University text, Springer (2008). 

\bibitem{MS} \text{A. S. Merkurjev  and A. A. Suslin}
  \textit{ K-cohomology of Severi-Brauer varieties and the norm residue homomorphism},
 Math. USSR Izv.,   {\bf 21} (1983), 307--340. 

\bibitem{MT} \text{ A. S. Merkurjev  and J.-P. Tignol},
  \textit{The multipliers of similitudes and the {B}rauer group of
              homogeneous varieties},
 J. Reine Angew. Math.,  {\bf 461} (1995), 13--47. 

\bibitem{Neukirch} \text{J. Neukirch, A. Schmidt and K. Wingberg}, \textit{Cohomology of Number Fields}, 
Second Edition, Springer Grundlehren der mathematischen Wissenshaften Bd.  {\bf 323}.


  \bibitem{ovv} \text{D. Orlov, A. Vishik and V. Voevodsky}, \textit{
An exact sequence for $K^M_\ast/2$ with applications to quadratic forms},
 Ann. of Math.   {\bf 165}  (2007),  1--13.

\bibitem{PP} \text{R. Parimala and R. Preeti},  \textit{Hasse principle for classical groups over function
 fields of curves over number fields}, J. Number Theory  {\bf 101 } (2003),  151--184.
 
 \bibitem{PPS} R. Parimala, R. Preeti and V. Suresh,
 \textit{Local-global principle for reduced norms over function fields of p-adic curves}, 
 Compos. Math. {\bf 154} (2018),   410--458.

 \bibitem{PSS} R. Parimala, R. Sridharan  and V. Suresh,  \textit
  {Hermitian analogue of a theorem of Springer},   Journal of Algebra {\bf 243} (2001), 780--789. 
     
 \bibitem{PS2022} R. Parimala  and V. Suresh,  \textit
  {Local-global principle for classical groups
over function fields of p-adic curves},    Comment. Math. Helv. {\bf 97} (2022), 255-304. 





\bibitem{preeti} \text{R. Preeti}, \textit{Classification theorems for Hermitian forms, 
the Rost kernel and Hasse principle
over fields with $cd_2(k) \leq 3$}, J. Algebra {\bf 385} (2013), 294--313.

\bibitem{reiner} \text{I. Reiner}, \textit{Maximal Orders}, London Math. Soc. Monogr. New Ser., 
Vol {\bf  28},  The Clarendon Press, Oxford University Press, Oxford 2003. 

\bibitem{roquette}  \text{Roquette, Peter},
 \textit{On the {G}alois cohomology of the projective linear group and
              its applications to the construction of generic splitting
              fields of algebras},
 {Math. Ann.}, {\bf 150} (1963),
 {411--439}.  

\bibitem{Sch}\text{W. Scharlau},  \textit{Quadratic and {H}ermitian forms},
  Grundlehren der mathematischen Wissenschaften  {\bf 270}, {Springer-Verlag},
  {Berlin} (1985). 
  
  
\bibitem{Scheiderer} \text{Scheiderer, Claus},
  \textit{Real and \'etale cohomology},  {Lecture Notes in Mathematics},
    {\bf 1588},
{Springer-Verlag, Berlin}, {1994}.  

\bibitem{SerreGC} \text{J-P. Serre}, \textit{Galois Cohomology}, 
Springer-Verlag, New York, 1997.

\bibitem{serrecohinv}
\text{J-P. Serre, Cohomological invariants, Witt invariants and trace forms}, in
\textit{ Cohomological Invariants in Galois Cohomology}, Skip Garibaldi, Alexander Merkurjev,
Jean-Pierre Serre, University Lecture Series {\bf 28}, Amer. Math. Soc. 2003.


\bibitem{Suresh2024} \text{V. Suresh}, \textit{Local-global principle for    groups    of type $A_n$ over semi global fields}, 
preprint (2024).


\bibitem{suslin}\text{A.A. Suslin},  \textit{Algebraic K -theory and the Norm-Residue Homomorphism}, 
J. Soviet Math. {\bf 30} (1985), 2556--2611.

 \bibitem{wads} \text{A. R. Wadsworth},
\textit{Merkurjev's elementary proof of {M}erkurjev's theorem},
{Applications of algebraic {$K$}-theory to algebraic geometry
              and number theory, {P}art {I}, {II} ({B}oulder, {C}olo.,
              1983)},  {Contemp. Math.},
 {\bf 55} (1986) {741--776}. 
  
 
 \end{thebibliography}
\end{document}